\documentclass[12pt,letterpaper,reqno]{amsart}

\usepackage{times}
\usepackage[T1]{fontenc}
\usepackage{mathrsfs}
\usepackage{latexsym}
\usepackage[dvips]{graphics}
\usepackage{epsfig}
\usepackage{hyperref}

\usepackage{amsmath,amsfonts,amsthm,amssymb,amscd}
\input amssym.def
\input amssym.tex

\newcommand{\Lchid}{L_E(s,\chi_d)}
\newcommand{\Lambdachid}{\Lambda(s,\chi_d)}
\newcommand{\Lambdachidone}{\Lambda(1-s,\chi_d)}
\newcommand{\Lambdaderivative}{\frac{\Lambda'(s,\chi_d)}{\Lambda(s,\chi_d)}}
\newcommand{\RRe}{\text{Re}}

\newcommand{\sym}{\mbox{\rm sym}}

\newcommand{\kommentar}[1]{}
\newcommand{\dalpha}{\frac{d}{d\alpha}}

\newcommand{\RR}{\mathbb{R}}

\newcommand{\NN}{\mathbb{N}}

\newcommand{\supp}{{\rm supp}}

\newcommand{\hg}{\widehat{g}}

\newcommand{\hphi}{\widehat{\phi}}  

\newcommand\be{\begin{equation}}
\newcommand\ee{\end{equation}}
\newcommand\bea{\begin{eqnarray}}
\newcommand\eea{\end{eqnarray}}
\newcommand\bi{\begin{itemize}}
\newcommand\ei{\end{itemize}}
\newcommand\ben{\begin{enumerate}}
\newcommand\een{\end{enumerate}}
\newcommand\bc{\begin{center}}
\newcommand\ec{\end{center}}
\newcommand\ba{\begin{array}}
\newcommand\ea{\end{array}}

\def\notdiv{\ \mathbin{\mkern-8mu|\!\!\!\smallsetminus}}

\newcommand{\R}{\ensuremath{\mathbb{R}}}
\newcommand{\C}{\ensuremath{\mathbb{C}}}

\newcommand{\fof}{\frac{1}{4}}  
\newcommand{\foh}{\frac{1}{2}}  

\newcommand\lag[2]{\ensuremath{\left(\frac{#1}{#2}\right)}}

\newtheorem{thm}{Theorem}[section]

\newtheorem{lem}[thm]{Lemma}

\theoremstyle{definition}
\newtheorem{rek}[thm]{Remark}

\newcommand{\twocase}[5]{#1 \begin{cases} #2 & \text{{\rm #3}}\\ #4
&\text{{\rm #5}} \end{cases}   }

\newcommand{\gep}{\epsilon}

\newcommand{\intii}{\int_{-\infty}^\infty}
\newcommand{\ci}{\frac1{2\pi i}}

\numberwithin{equation}{section}

\begin{document}

\title{An elliptic curve test of the $L$-Functions Ratios Conjecture}

\author{Duc Khiem Huynh}\email{dkhuynhms@gmail.com}\address{Department of Pure Mathematics, University of Waterloo, Waterloo, ON, N2L 3G1, Canada}

\author{Steven J. Miller}\email{Steven.J.Miller@williams.edu}\address{Department of Mathematics and Statistics, Williams College, Williamstown, MA 01267}

\author{Ralph Morrison}\email{morrison@math.berkeley.edu}\address{Department of Mathematics and Statistics, Williams College, Williamstown, MA 01267}\curraddr{Department of Mathematics, University of California at Berkeley, Berkeley, CA 94708}


\subjclass[2010]{11M41 (primary), 15B52 (secondary).}

\keywords{$1$-Level Density, Dirichlet $L$-functions, Low Lying
Zeros, Ratios Conjecture, Twists of an Elliptic Curve}

\date{\today}

\thanks{
The first named author was partially supported by EPSRC, a CRM postdoctoral fellowship and NSF grant DMS0757627.
The second named author was partially supported by NSF grants DMS0855257 and DMS0970067. The third named author was partially supported by an NSF Graduate Fellowship.}

\begin{abstract} We compare the $L$-Function Ratios Conjecture's prediction with number theory for the family of quadratic twists of a fixed elliptic curve with prime conductor, and show agreement in the 1-level density up to an error term of size $X^{-\frac{1-\sigma}{2}}$ for test functions supported in $(-\sigma,\sigma)$; this gives us a power-savings for $\sigma<1$. This test of the Ratios Conjecture introduces complications not seen in previous cases (due to the level of the elliptic curve). Further, the results here are one of the key ingredients in the companion paper \cite{DHKMS2}, where they are used to determine the effective matrix size for modeling zeros near the central point for this family. The resulting model beautifully describes the behavior of these low lying zeros for finite conductors, explaining the data observed by Miller in \cite{Mil3}.

A key ingredient in our analysis is a generalization of Jutila's bound for sums of quadratic characters with the additional restriction that the fundamental discriminant be congruent to a non-zero square modulo a square-free integer $M$. This bound is needed for two purposes. The first is to analyze the terms in the explicit formula corresponding to characters raised to an odd power. The second is to determine the main term in the 1-level density of quadratic twists of a fixed form on ${\rm GL}_n$. Such an analysis was performed by Rubinstein \cite{Rub}, who implicitly assumed that Jutila's bound held with the additional restriction on the fundamental discriminants; in this paper we show that assumption is justified.
\end{abstract}


\maketitle

\setcounter{equation}{0}

\tableofcontents


\section{Introduction}

One of the most important areas in modern number theory is the study of the distribution of the zeros of $L$-functions.
These zeros encode crucial number theoretic information on subjects ranging from the distribution of the primes (from simply the number of primes at most $x$ to biases in the distribution of primes in various residue classes) to properties of class numbers to (conjecturally) the geometric rank of the Mordell-Weil group of rational solutions of an elliptic curve. Further, the observed behavior is similar to that found in nuclear physics and other disciplines, suggesting deep connections between this branch of mathematics and other fields.  The General Riemann Hypothesis (GRH), often considered the most important open question in mathematics, is the conjecture that all non-trivial zeros of these $L$-functions have real part equal to $1/2$. As powerful as this conjecture is, there are many problems in number theory where just knowing the real parts are $1/2$ is not enough, and we need to know finer properties of the distribution of the zeros on the critical line $\Re(s) = 1/2$.

As proofs of properties of these zeros have eluded researchers since Riemann's seminal paper, methods of modeling these zeros are indispensable in understanding and formulating appropriate conjectures about $L$-functions. Many models have had various degrees of success. Perhaps the most famous are those arising from Random Matrix Theory (see for example \cite{KaSa1,KaSa2,KeSn1,KeSn2,KeSn3} among others, and \cite{FM} for some of the history of the interplay between nuclear physics and number theory). Unfortunately, these models are only able to predict the main term behavior in the problems of interest, and in many situations the arithmetic of the family of $L$-functions only surfaces in lower order terms (see for instance \cite{Mil2,Mil6,Yo1}). This often requires the arithmetic to be added in an ad-hoc fashion. Another approach, which has the advantage of including the arithmetic directly, is the hybrid model (see \cite{GHK}), where $L$-functions are modeled by the product of a partial Hadamard product of zeros (which is expected to be described by Random Matrix Theory) and a partial Euler product (which is expected to provide the arithmetic).

In this work we discuss another method, the $L$-function Ratios Conjecture of Conrey, Farmer and Zirnbauer \cite{CFZ1,CFZ2}. We concentrate on the family of quadratic twists of a fixed elliptic curve of prime conductor. The paper is organized as follows. We first describe the statistic of interest (the one-level density), and then discuss the Ratios Conjecture's prediction and its implications. The rest of the paper is devoted to proving the conjecture. We calculate the number theory in \S\ref{sec:numbertheory}, and show for suitable test functions that it agrees with the Ratios' prediction in \S\ref{derivativeofAE}. A key step in the analysis is generalizing Jutila's bound for character sums, which we do in \S\ref{sec:jutila}. In addition to being of use for this problem, this result was also implicitly used by Rubinstein \cite{Rub} in determining the main term in the one-level density for twists of a fixed ${\rm GL}_n$ form.

\subsection{One-Level Density of Low Lying Zeros}

Assuming GRH, the non-trivial zeros of $L$-functions lie on the critical line, and thus it makes sense to study the distribution of spacings. There is a mix of theoretical and experimental evidence (\cite{Mon,Hej,RS,Od1,Od2}) relating these normalized spacings in the limit as we climb the critical line to the scaled spacings between eigenvalues of random matrix ensembles as the matrix size tends to infinity. Initially this suggested that the Gaussian Unitary Ensemble (GUE) of matrices was the correct (and only) model needed for number theory; however, Katz and Sarnak showed that the classical compact groups (subgroups of $N\times N$ unitary matrices) all have the same $n$-level correlations as the GUE as $N\to\infty$. There is thus more to the story, and we need a statistic which is sensitive to finer properties of the $L$-functions.

One such statistic is the one-level density of the low lying zeros of a family of $L$-functions, which is different for the scaling limits of the different classical compact groups. Fix a Schwartz test function $\phi$ such that $\hphi$ is supported in, say, $(-\sigma, \sigma)$. Let $L$ be related to the local rescaling near the central point, so that normalized zeros near $s=1/2$ have mean spacing one. For an $L$-function $L(s,f)$, its one-level density is
defined by
\be D(f,\phi)\ := \ \sum_{\gamma_f}\phi\left(\frac{\gamma_f L}{\pi}\right);
\ee here $1/2 + i\gamma_f$ runs over the non-trivial zeros of the $L$-function (which under GRH all have $\gamma \in \R$) and $L/\pi$ is the scaling factor (it is related to the logarithm of the analytic conductor).\footnote{Many works in the literature use $L'/2\pi$; as this is a companion paper to \cite{HKS} we use their notation to facilitate calling their equations.} Using the explicit formula (see for instance \cite{Mes,RS}), we replace the sum of $\phi$ at the scaled zeros with sums of $\hphi$ at the logarithms of the primes, weighted by the Fourier coefficients of the $L$-function. As $\phi$ is a Schwartz function, it vanishes rapidly as $|x|\rightarrow\infty$ and thus most of the contribution is from zeros near the central point (relative to the local average spacing).

Ideally we would use a delta spike instead of a Schwartz test function to get a perfect picture at a point;  however, the delta spike has a Fourier transform of infinite support, which leads to weighted prime sums we cannot evaluate.  As each $L$-function only has a bounded number of zeros within the average spacing of the central point, it is necessary to average the one-level density over all $f$ in a family $\mathcal{F}$. This allows us to use results from number theory\footnote{The needed result depends of course on the family being studied. For Dirichlet $L$-functions one uses the orthogonality of the characters, for elliptic curves one uses properties of sums of Legendre symbols, while for cuspidal newforms one uses the Petersson formula.} to determine the behavior on average near the central point. The exact nature of just what constitutes a family is still being determined; standard examples include $L$-functions attached to Dirichlet characters, cuspidal newforms, and families of elliptic curves to name just a few.

We assume our family of $L$-functions $\mathcal{F}$ can be ordered by conductor, and denote by $\mathcal{F}(Q)$ all elements of the family whose conductor is at most $Q$. Thus the quantity of interest ends up being
\be D(\mathcal{F},\phi) \ := \ \lim_{Q\to\infty}  \frac{1}{|\mathcal{F}(Q)|}\sum_{f\in\mathcal{F}(Q)} D(f,\phi) \ = \ \lim_{Q\to\infty} \frac{1}{|\mathcal{F}(Q)|}\sum_{f\in\mathcal{F}(Q)}\sum_{\gamma_f}\phi\left(\frac{\gamma_f L}{\pi}\right).
\ee
In other words, we consider the limiting behavior of the average of the one-level densities as the conductors grow. To date a large number of families have been investigated (such as Dirichlet $L$-functions, elliptic curves, cuspidal newforms, symmetric powers, number fields, and convolutions of such families, to name a few), and for suitably restricted test functions the main terms in the one-level densities agree with the scaling limits of a classical compact group; see for example \cite{DM1,DM2,FI,Gao,Gu,HM,HR,ILS,Mil1,OS1,OS2,RR,Ro,Rub,Yo2}.

\subsection{The Ratios Conjecture}

While Random Matrix Theory has successfully predicted the main term of the one-level density of all families studied to date, it is insufficient as it is silent on lower order terms. These terms are important for many reasons. The first is that the arithmetic of the family is often absent in the main term but present in lower order terms (see for instance \cite{Mil2,Mil6,Yo1}). For example, in \cite{Mil6} lower order effects were found related to the torsion group of the family of elliptic curve $L$-functions. Further, these lower order terms are important, as they control the rate of convergence to the predicted limiting behavior. This work is motivated by the companion paper \cite{DHKMS2}. The authors there discuss a proposed model which explains the observed repulsion found by Miller \cite{Mil3} of zeros of elliptic curve $L$-functions near the central point. One of the two main ingredients in the model is the first lower order term in the one-level density in elliptic curve families, which is needed to determine the effective matrix size. The Ratios' prediction of this was worked out in another companion paper, \cite{HKS}; the purpose of this paper is to verify the Ratios' prediction (at least for suitably restricted support).

The $L$-function Ratios Conjecture of Conrey, Farmer and Zirnbauer \cite{CFZ1,CFZ2} (see also \cite{CS1} for many worked out examples of the conjecture's prediction) are formulas for the averages over families of $L$-functions of ratios of products of shifted $L$-functions.  Their ``recipe'' for performing these calculations starts by using the approximate functional equation, where the error term is discarded, to expand the $L$-functions in the numerator; the $L$-functions in the denominator are expanded via the Mobius function. They then average over the family, and retain only the diagonal pieces. These are restricted sums over integers, but are then completed and extended to sums over all integers; again the error term introduced is ignored.  These methods, far simpler to implement than rigorous analysis, have easily predicted the answers to many difficult computations, and have shown remarkable accuracy. The resulting formulas make very detailed predictions on numerous problems, ranging from moments to spacings between
adjacent zeros and values of $L$-functions.

A standard test of the Ratios Conjecture is to compare the Ratios Conjecture's predictions for the one-level density  of a family of $L$-functions with the corresponding rigorous calculation. Agreement has been found for suitably restricted test functions for many families. See \cite{CS1,GJMMNPP,Mil3,Mil5,Mil6,MilMon}, as well as \cite{BCY,CS1,CS2} for agreement with other statistics.  In addition to strengthening the credibility of the conjecture, these calculations provide insight into the significance of the terms that arise in the number theoretic calculations whose corresponding terms in the Ratios Conjecture's predictions are more clearly understandable. For example, in \cite{Mil5} the Ratios Conjecture's prediction allows the interpretation of a lower order term in the behavior of the family of quadratic Dirichlet characters as arising from the non-trivial zeros of the Riemann zeta function.

Our primary object of study is the collection of quadratic twists of a fixed elliptic curve of prime conductor $M$. The families associated to elliptic curves are of considerable importance, as they are the best laboratories (see \cite{Mil3}) to see the effect of multiple zeros on nearby zeros. By work of C. Breuil, B. Conrad, F. Diamond. R. Taylor and A. Wiles \cite{BCDT,TW,Wi}, the $L$-function of an elliptic curve agrees with that of a weight 2 cuspidal newform of level $N$ (where the integer $N>1$ is the conductor of the elliptic curve). The Ratios' prediction was computed in \cite{HKS}, and was one of the key inputs in \cite{DHKMS2} in explaining the observed repulsion of zeros near the central point in families of elliptic curve $L$-functions (see \cite{DHKMS1} for an analysis of random matrix quantities relevant for the model and comparison).  We perform the number theoretic calculations of the zero statistics for the one-level density for this family, and compare our results to the Ratios Conjecture's prediction.  For a similar case see \cite{MilMor}, which performed comparable calculations for the family of quadratic twists of the $L$-function associated to Ramanujan's tau function, and found agreement with the Ratios' prediction up to a power-savings error term.  These $L$-functions are similar to our elliptic curve $L$-functions but without the bad prime.  The simpler case provided a useful guide for performing the more complicated analysis found in this paper.

We first set some notation for the paper. We always denote our elliptic curve by $E$, which we assume has prime conductor $M$ and even functional equation. We consider the family of quadratic twists,

\begin{equation}\label{eq:familydefinition}
{\mathcal{F}(X)} = \left\{0 < d \leq X : d \mbox{~an even fundamental discriminant and~} \chi_d(-M)\omega_E = 1 \right\}
\end{equation}
and set
\begin{equation}\label{eq:scalingdefinition}
X^\ast\ =\ |\mathcal{F}(X)|, \quad\quad L = \log\left(\frac{\sqrt{M}X}{2\pi} \right).
\end{equation}

The Ratios Conjecture's prediction for these lower order terms, computed in \cite{HKS}, has been inputted in some of these models, but has not yet been verified. The main obstacle in verifying the prediction, at least for suitably restricted test functions, is the presence of the level $M$ in the Euler products in the prediction. This leads to more complicated formulas than in \cite{Mil5}, where we studied just quadratic Dirichlet characters. While the resulting Euler products are harder to analyze than other cases, we are still able to show agreement with a power savings.

Our main (number theory) result is the following:
\begin{thm}\label{thm:numbertheoryresult} Let $E$ be an elliptic curve with even functional equation and prime conductor $M$ and $g$ an even Schwartz test function whose Fourier transform $\widehat{g}$ is supported in $(-\sigma, \sigma)$. The one-level density of the family of even quadratic twists of $E$ by even fundamental discriminants at most $X$ is
\begin{eqnarray}\label{eq:numbtheoryresulteq}
& &  \frac{1}{X^*}\sum_{d \in \mathcal{F}(X)}\sum_{\gamma_d} g\left(\gamma_d \frac{L}{\pi}\right)  \nonumber\\
& =\ & \frac{g(0)}{2}+ \frac{1}{2 LX^*} \int_{-\infty}^\infty g(\tau) \sum_{d \in \mathcal{F}(X)}
\left[2\log\left(\frac{\sqrt{M}|d|}{2 \pi} \right)
+ \frac{\Gamma'}{\Gamma}\left(1 + i\frac{\pi \tau}{L}\right) +\frac{\Gamma'}{\Gamma}\left(1 - i\frac{\pi \tau}{L}\right) \right] d \tau\nonumber\\
& & \ + \frac{1}{L}\intii
g(\tau) \left(-\frac{\zeta'}{\zeta}\left(1+\frac{2\pi i \tau}{L}\right)+ \frac{L_E'}{L_E}\left({\rm sym}^2,1+\frac{2\pi i \tau}{L}\right)-\sum_{\ell=1}^\infty\frac{(M^\ell-1)\log M}{M^{\left(2+\frac{2\pi i\tau}{L}\right)\ell}}\right) d\tau\nonumber\\
& & \ -\frac{1}{L} \sum_{k=0}^\infty \int_{-\infty}^\infty g(\tau) \frac{\log M}{M^{(k+1)(1 + \frac{\pi i \tau}{L})}} d \tau+\frac{1}{L} \int_{-\infty}^\infty g(\tau) \sum_{p\nmid M} \frac{\log p}{(p + 1)} \sum_{k = 0}^\infty \frac{\lambda(p^{2k + 2}) - \lambda(p^{2k}) }{p^{(k + 1)(1 + \frac{2 \pi i \tau}{L})}}\ d\tau\nonumber\\
& & \ + O_M\left(X^{-\frac{1-\sigma}2}  \log^6 X\right).
\end{eqnarray}
\end{thm}

Much of the work in determining the Ratios' prediction was done in \cite{HKS}. In this work we finish the analysis, rewriting the expansion from \cite{HKS} to facilitate comparisons with number theory.

\begin{thm}\label{thm:agreement} Notation as in Theorem \ref{thm:numbertheoryresult}, the prediction from the Ratios Conjecture is
\begin{eqnarray} \nonumber
&& \frac{1}{X^*}\sum_{d \in \mathcal{F}(X)} \sum_{\gamma_d}g\Big(\frac{\gamma_d
L}{\pi}\Big)
\\ \nonumber
&& =~ \frac{1}{2LX^*} \int_{-\infty}^\infty g(\tau) \sum_{d \in \mathcal{F}(X)}
\Bigg[
2\log
\left(\frac{\sqrt{M}|d|}{2\pi} \right) +
\frac{\Gamma'}{\Gamma}\Big(1 + \frac{i\pi \tau}{L}
\Big) + \frac{\Gamma'}{\Gamma}\Big(1 - \frac{i \pi \tau}{L}\Big)
\Bigg]
d\tau
\\ \nonumber
&& + \frac{1}{L}\intii
g(\tau) \left(-\frac{\zeta'}{\zeta}\left(1+\frac{2\pi i \tau}{L}\right)+ \frac{L_E'}{L_E}\left({\rm sym}^2,1+\frac{2\pi i \tau}{L}\right)-\sum_{\ell=1}^\infty\frac{(M^\ell-1)\log M}{M^{\left(2+\frac{2i\pi\tau}{L}\right)\ell}}\right) d\tau\\ \label{formularfinalthm}
&& ~~~-\frac{1}{L} \sum_{k=0}^\infty \int_{-\infty}^\infty g(\tau) \frac{\log M}{M^{(k+1)(1 + \frac{\pi i \tau}{L})}} d \tau+\frac{1}{L} \int_{-\infty}^\infty g(\tau) \sum_{p\nmid M} \frac{\log p}{(p + 1)} \sum_{k = 0}^\infty \frac{\lambda(p^{2k + 2}) - \lambda(p^{2k}) }{p^{(k + 1)(1 + \frac{2 \pi i \tau}{L})}} d\tau\nonumber\\
&& -\frac{1}{LX^*} \int_{-\infty}^\infty g(\tau) \sum_{d \in \mathcal{F}(X)}
\Bigg[
\bigg(\frac{\sqrt{M}|d|}{2\pi}\bigg)^{-2i \pi \tau / L}
\frac{\Gamma(1 - \frac{i \pi \tau}{L})}{\Gamma(1 + \frac{i \pi
\tau}{L})} \frac{\zeta(1 + \frac{2 i \pi \tau}{L})L_E(\mbox{\rm \sym}^2, 1 -
\frac{2 i \pi \tau}{L})}{L_E(\mbox{\rm \sym}^2,1)}\nonumber\\
&& \times A_E\Big(-\frac{i \pi
\tau}{L}, \frac{i \pi \tau}{L}\Big)
\Bigg]d\tau  \label{predictionfinalthm}
 + O(X^{-1/2+\varepsilon});
\end{eqnarray} see \S\ref{derivativeofAE} for a definition of $A_E$.
\end{thm}

A mentioned above, the main difficulty in showing agreement between number theory and the above prediction is the presence of the level of the elliptic curve (which was not present in the symplectic family studied in \cite{Mil5}). By a careful analysis of the Euler products, we prove

\begin{thm}\label{thm:ntequalsratios} Notation as in Theorem \ref{thm:numbertheoryresult}, assuming GRH the Ratios Conjecture's prediction agrees with number theory for ${\rm supp}(\hphi) \subset (-\sigma, \sigma)$, up to error terms of size $O(X^{-(1-\sigma)/2})$.
\end{thm}


\section{The Number Theory Result}\label{sec:numbertheory}

The starting point of all one-level density investigations is the explicit formula; modifying \cite{Mes,RS} (among others; see Appendix \ref{sec:expformula} for a proof) one finds the following:

\begin{lem} The one-level density for the family of quadratic twists by even fundamental discriminants of a fixed elliptic curve $E$ with even functional equation and prime conductor $M$ is
\begin{align} \nonumber
& \frac{1}{X^*}\sum_{d \in \mathcal{F}(X)}\sum_{\gamma_d} g\left(\gamma_d \frac{L}{\pi}\right) \\ \nonumber
& = \frac{1}{2 LX^*} \int_{-\infty}^\infty g(\tau) \sum_{d \in \mathcal{F}(X)}
\left[2\log\left(\frac{\sqrt{M}|d|}{2 \pi} \right)
+ \frac{\Gamma'}{\Gamma}\left(1 + i\frac{\pi \tau}{L}\right) +\frac{\Gamma'}{\Gamma}\left(1 - i\frac{\pi \tau}{L}\right) \right] d \tau\\
& \ \ \ \ - \ \frac{2}{2 L} \sum_{d \in \mathcal{F}(X)} \sum_{k = 1}^\infty \sum_p \frac{(\alpha_p^k + \beta_p^k)\chi_d^k(p)\log p}{p^{k/2}}
\widehat{g}\left(\frac{\log p^k}{2 L} \right),
\end{align}
where ${\mathcal{F}(X)}$, $X^*$, and $L$ are as defined in Equations \ref{eq:familydefinition} and \ref{eq:scalingdefinition}.
\end{lem}

We prove Theorem \ref{thm:numbertheoryresult} by analyzing the expansion above. As the integral term is also found in the Ratios' prediction, we need only study \bea
S & \ = \ & -\frac{2}{2 L X^*} \sum_{d \in \mathcal{F}(X)} \sum_{k = 1}^\infty \sum_p \frac{(\alpha_p^k + \beta_p^k)\chi_d^k(p)\log p}{p^{k/2}}
\widehat{g}\left(\frac{\log p^k}{2 L} \right)\ = \  S_{\rm even} + S_{\rm odd},
\eea
where
\bea
S_{\rm even}  &\ = \ &-\frac{1}{X^*} \sum_{d\in\mathcal{F}( X)} \sum_{k = 1}^\infty \sum_p \frac{(\alpha_p^{2k} + \beta_p^{2k})\chi_d^2(p) \log p}{p^k L} \widehat{g}\left(\frac{\log p^k}{L} \right)
\nonumber\\
S_{\rm odd} &=& -\frac{1}{X^*} \sum_{d\in\mathcal{F}( X)} \sum_{k = 0}^\infty \sum_p \frac{(\alpha_p^{2k + 1} + \beta_p^{2k +1})\chi_d(p) \log p}{p^{(2k+1)/2} L} \widehat{g}\left(\frac{\log p^{2k+1}}{2L} \right)
\eea
(note that $\chi_d(p) = \chi_d^{2k+1}(p)$ for any $k \in \NN$). We split $S_{\rm even}$ further by noting that
\be
\twocase{\chi_d^2(p) \ = \ }
{1}{if $p \nmid d$}{0}{if $p|d$},
\ee
and write
\be
S_{\rm even}\ =\ S_{\rm even,1} + S_{\rm even,2}
\ee
with
\bea
S_{\rm even,1} &\ =\ & -\sum_p \sum_{k = 1}^\infty  \frac{(\alpha_p^{2k} + \beta_p^{2k}) \log p}{p^k L} \widehat{g}\left(\frac{\log p^k}{L} \right)
\nonumber\\ S_{\rm even,2} &=& \frac{1}{X^*} \sum_{d\in\mathcal{F}( X)}\sum_{k = 1}^\infty \sum_{p|d} \frac{(\alpha_p^{2k} + \beta_p^{2k}) \log p}{p^k L} \widehat{g}\left(\frac{\log p^k}{L} \right).
\eea

We prove Theorem \ref{thm:numbertheoryresult} by analyzing $S_{{\rm even}}$ and $S_{{\rm odd}}$ in a series of lemmata below, frequently breaking these summands down further.

\subsection{Analysis of $S_{\rm even,1}$}

We consider $S_{\rm even, 1}$ and have
\begin{align} \nonumber
S_{\rm even,1} & = -\frac{1}{L}\sum_p \sum_{k = 1}^\infty  \frac{(\alpha_p^{2k} + \beta_p^{2k}) \log p}{p^k} \widehat{g}\left(\frac{\log p^k}{L} \right) = S_{\rm even,1,1} + S_{\rm even,1,2},
\end{align}
where
\begin{eqnarray}
S_{\rm even,1,1} & \ = \ & -\frac{1}{L} \sum_{k = 1}^\infty \frac{(\alpha_M^{2k} + \beta_M^{2k}) \log M}{M^k}
\widehat{g}\left(\frac{\log M^k}{L} \right)
\nonumber\\ S_{\rm even,1,2} &=&
-\frac{1}{L} \sum_{p\nmid M}\sum_{k = 1}^\infty \frac{(\alpha_p^{2k} + \beta_p^{2k}) \log p}{p^k}\widehat{g}\left(\frac{\log p^k}{L} \right).
\end{eqnarray}

\begin{lem} We have
\begin{equation} \label{Sevenoneonefinal}
S_{\rm even,1,1} = -\frac{1}{L} \sum_{k=1}^\infty \frac{\log M}{M^{2k}} \widehat{g}\left(\frac{\log M^k}{L} \right)
= -\frac{1}{L} \sum_{k=1}^\infty \int_{-\infty}^\infty g(\tau) \frac{\log M}{M^{2k(1 + \frac{\pi i \tau}{L})}} d \tau.
\end{equation}
\end{lem}

\begin{proof} For $M$ we have
\begin{equation} \label{alphabetabad}
\alpha_M^{2k} + \beta_M^{2k} \ =\ \left(\frac{\omega_E}{M^{1/2}}\right)^{2k} \ = \ M^{-k}.
\end{equation}
Using (\ref{alphabetabad}) and unwinding the Fourier transform gives the claim. \end{proof}

\begin{lem}\label{lem:ntsumkevenEC1} Notation as above, \bea &&S_{{\rm even},1,2} = \nonumber\\
&& \frac{g(0)}2 + \frac{1}{L}\intii
g(\tau) \left(- \frac{\zeta'}{\zeta}\left(1+\frac{2\pi i \tau}{L}\right)+ \frac{L_E'}{L_E}\left({\rm sym}^2,1+\frac{2\pi i \tau}{L}\right)-\sum_{\ell=1}^\infty\frac{(M^\ell-1)\log M}{M^{\left(2+\frac{2\pi i\tau}{L}\right)\ell}} \right) d\tau.
\nonumber\\ \eea
\end{lem}

\begin{proof}[Proof of Lemma \ref{lem:ntsumkevenEC1}] Let \be \twocase{\Lambda_E(n)\ =\ }{(\alpha_p^{2\ell} + \overline{\alpha}_p^{2\ell}) \log p}{if $n=p^\ell$, $p\nmid M$}{0}{otherwise.} \ee

We have \be S_{{\rm even};1,2} \ = \ -\frac{1}{L}\sum_{n=1}^\infty \frac{\Lambda_E(n)}{n} \ \hg\left(\frac{\log
n}{L}\right). \ee We use Perron's formula to re-write $S_{{\rm
even};1}$ as a contour integral. For any $\gep > 0$ set \bea\label{eq:integraldefI1} I_1 \ =
\ \ci \int_{\Re(z)=1+\gep} g\left(\frac{(2z-2)\log A}{4\pi i}\right)
\sum_{n=1}^\infty \frac{\Lambda_E(n)}{n^z}\ dz; \eea we will later
take $A =\sqrt{M}X/2\pi$, so that $\log A=L$. We write $z = 1+\gep+iy$ and use
\eqref{eq:extendingphi} (replacing $\phi$ with $g$) to write
$g(x+iy)$ in terms of the integral of $\hg(u)$. We have \bea I_1 & \
= \ & \sum_{n=1}^\infty \frac{\Lambda_E(n)}{n^{1+\gep}} \ci \intii
g\left(\frac{y\log A}{2\pi}-\frac{i\gep\log
A}{2\pi}\right)e^{-iy\log n} idy \nonumber\\ &=& \sum_{n=1}^\infty
\frac{\Lambda_E(n)}{n^{1+\gep}} \frac1{2\pi} \intii \left[\intii
\left[\hg(u)e^{\gep u \log A}\right] e^{-2\pi i \frac{-y\log A}{2\pi
}u} du \right]e^{-iy\log n}dy.\ \ \ \ \ \ \eea We let $h_\gep(u) =
\hg(u) e^{\gep u \log A}$. Note that $h_\gep$ is a smooth, compactly
supported function and $\widehat{\widehat{h_\gep}}(w) = h_\gep(-w)$.
Thus \bea I_1 & \ = \ & \sum_{n=1}^\infty
\frac{\Lambda_E(n)}{n^{1+\gep}}  \frac1{2\pi} \intii
\widehat{h_\gep}\left(-\frac{y\log A}{2\pi}\right) e^{-iy\log n} dy
\nonumber\\ &=& \sum_{n=1}^\infty \frac{\Lambda_E(n)}{n^{1+\gep}}
\frac1{2\pi}\intii \widehat{h_\gep}(y) e^{-2\pi i \frac{-y \log
n}{\log A}}\ \frac{2\pi dy}{\log A} \nonumber\\ &=&
\sum_{n=1}^\infty \frac{\Lambda_E(n)}{n^{1+\gep}} \frac1{\log A}\
\widehat{\widehat{h_\gep}}\left(-\frac{\log n}{\log A}\right)
\nonumber\\
&=& \sum_{n=1}^\infty \frac{\Lambda_E(n)}{n^{1+\gep}}  \frac1{\log A}\
\hg\left(\frac{\log n}{\log A}\right) e^{\gep \log n} \nonumber\\
&=& \frac1{\log A}\sum_{n=1}^\infty \frac{\Lambda_E(n)}{n}\
\hg\left(\frac{\log n}{\log A}\right). \eea By taking $A = \sqrt{M}X/2\pi$
we find \be S_{{\rm even};1,2} \ = \ -\frac{1}{L}\sum_{n=1}^\infty \frac{\Lambda_E(n)}{n} \ \hg\left(\frac{\log
n}{L}\right) \ = \ -I_1.\ee

We now re-write $I_1$ by shifting contours; we will not pass any
poles as we shift. For each $\delta > 0$ we consider the contour
made up of three pieces: $(1-i\infty,1-i\delta]$, $C_\delta$, and
$[1-i\delta,1+i\infty)$, where $C_\delta = \{z: z-1 = \delta
e^{i\theta}, \theta \in [-\pi/2,\pi/2]\}$ is the semi-circle going
counter-clockwise from $1-i\delta$ to $1+i\delta$. By Cauchy's
residue theorem, we may shift the contour in $I_1$ from $\Re(z) =
1+\gep$ to the three curves above.

Before analyzing this integral, we rewrite $\sum_n \Lambda_E(n)n^{-z}$ as the sum of logarithmic derivatives of $L$-functions.
From (3.15) and (3.16) of \cite{ILS}, we have
 \be L_E({\rm sym}^2,s) \ = \ \prod_{p\nmid M} \left(1 - \frac{\alpha_p^2}{p^s}\right)^{-1} \left(1 - \frac{1}{p^s}\right)^{-1} \left(1 - \frac{{\beta}^2_p}{p^s}\right)^{-1}\prod_{p|M}\left(1-\frac{1}{p^{s+1}}\right)^{-1}, \ee
  as $\alpha_p \beta_p = 1$ for $p\nmid M$.  Taking the logarithmic derivative yields
  \bea \frac{L_E'}{L_E}({\rm sym}^2,s) & \ = \ & -\sum_{p\nmid M}\sum_{\ell=1}^\infty \frac{(\alpha_p^{2\ell} + 1+\beta_p^{2\ell})\log p}{p^{s\ell}}-\sum_{p|M}\sum_{\ell=1}^\infty\frac{\log p}{p^{(s+1)\ell}}\nonumber
  \\ & \ = \ & -\sum_{p\nmid M}\sum_{\ell=1}^\infty \frac{(\alpha_p^{2\ell} +\beta_p^{2\ell})\log p}{p^{s\ell}}-\sum_{p\nmid M}\sum_{\ell=1}^\infty\frac{\log p}{p^{s\ell}}-\sum_{p|M}\sum_{\ell=1}^\infty\frac{\log p}{p^{(s+1)\ell}}\label{logderivativesym}, \eea
  so
\bea \sum_{n=1}^\infty\Lambda(n)n^{-s} & \ = \ & \sum_{p\nmid M}\sum_{\ell=1}^\infty \frac{(\alpha_p^{2\ell} + \beta_p^{2\ell})\log p}{p^{s\ell}}\nonumber
\\ & \ = \ &-\sum_{p\nmid M}\sum_{\ell=1}^\infty\frac{\log p}{p^{s\ell}}-\sum_{p|M}\sum_{\ell=1}^\infty\frac{\log p}{p^{(s+1)\ell}} -\frac{L_E'}{L_E}({\rm sym}^2,s)\nonumber
\\ & \ = \ &\frac{\zeta'}{\zeta}(s)-\frac{L_E'}{L_E}({\rm sym}^2,s)+\sum_{p|M}\sum_{\ell=1}^\infty\frac{\log p}{p^{s\ell}}-\sum_{p|M}\sum_{\ell=1}^\infty\frac{\log p}{p^{(s+1)\ell}} \nonumber\
\\ & \ = \ &\frac{\zeta'}{\zeta}(s)-\frac{L_E'}{L_E}({\rm sym}^2,s)+\sum_{\ell=1}^\infty\frac{(M^\ell-1)\log M}{M^{(s+1)\ell}}
\label{sumoflogderivs}.
\eea

We use this in replacing $\sum_n \Lambda_E(n)n^{-z}$ in the integral definition of $I_1$ in \eqref{eq:integraldefI1}. We find \bea I_1 & \ = \ &
\ci\left[\int_{1-i\infty}^{1-i\delta} + \int_{C_\delta} +
\int_{1+i\delta}^{1+i\infty} g\left(\frac{(2z-2)\log A}{4\pi
i}\right) \sum_n\frac{\Lambda_E(n)}{n^z}\ dz\right]  \nonumber\\ &=&
 \ci\left[\int_{1-i\infty}^{1-i\delta} + \int_{C_\delta} +
\int_{1+i\delta}^{1+i\infty} g\left(\frac{(2z-2)\log A}{4\pi
i}\right) \right. \nonumber\\ & & \ \ \ \ \ \ \ \ \cdot \left. \left(\frac{\zeta'}{\zeta}(z)-\frac{L_E'}{L_E}({\rm sym}^2,z) +\sum_{\ell=1}^\infty\frac{(M^\ell-1)\log M}{M^{(z+1)\ell}} \right) \ dz\right].  \eea

The integral over $C_\delta$ is easily evaluated. Shimura \cite{Sh} proved that $L_E({\rm sym}^2 ,s)$ is entire, and thus so too is its logarithmic derivative. Thus there is no contribution from the symmetric square piece in the limit as $\delta \to 0$. As $\zeta(s)$ has a pole at $s=1$, $\zeta'(s)/\zeta(s) = -1/(s-1) + \cdots$, and we must multiply the contribution from the residue by $-1$ because of the pole. We get just minus half the residue of $g\left(\frac{(2z-2)\log
A}{4\pi i}\right)$, which yields the contribution from the $C_\delta$ piece is
$-g(0)/2$.

We now take the limit as $\delta \to 0$: \bea I_1 & \ = \ &
-\frac{g(0)}2 - \lim_{\delta \to 0} \frac1{2\pi}
\left[\int_{-\infty}^{-\delta} + \int_{\delta}^\infty
g\left(\frac{y\log A}{2\pi}\right) \right. \nonumber\\ & & \ \ \ \ \ \ \ \ \ \ \cdot \left.
\left( - \frac{\zeta'}{\zeta}(z)+\frac{L_E'}{L_E}({\rm sym}^2,z)-\sum_{\ell=1}^\infty\frac{(M^\ell-1)\log M}{M^{(z+1)\ell}} \right) dy\right].\eea As $g$ is an even
Schwartz function, the limit of the integral above is well-defined
(for large $y$ this follows from the decay of $g$, while for small
$y$ it follows from the fact that $\zeta'(1+iy)/\zeta(1+iy)$ has a
simple pole at $y=0$ and $g$ is even). We again take $A=\sqrt{M}X/2\pi$,
and change variables to $\tau = yL/2\pi$. Thus \bea I_1 & \ = \ & -\frac{g(0)}2 - \frac{1}{L}\intii
g(\tau) \left( -\frac{\zeta'}{\zeta}\left(1+\frac{2\pi i \tau}{L}\right)+ \frac{L_E'}{L_E}\left({\rm sym}^2,1+\frac{2\pi i \tau}{L}\right)\right.\nonumber\\ & & \left. \ \ \ \ \ \ \ \ \ \ \ \ \ \ \ \ \ -\sum_{\ell=1}^\infty\frac{(M^\ell-1)\log M}{M^{\left(2+\frac{2\pi i\tau}{L}\right)\ell}} \right) d\tau \nonumber\\ &=& -S_{{\rm even},1,2}, \eea which completes the proof of Lemma
\ref{lem:ntsumkevenEC1}. \end{proof}

\subsection{Analysis of $S_{\rm even,2}$}

\begin{lem} We have
\begin{align} \nonumber
S_{\rm even,2} \ = \  \frac{1}{L} \int_{-\infty}^\infty g(\tau) \sum_{p\nmid M} \frac{\log p}{(p + 1)} \sum_{k = 0}^\infty \frac{\lambda(p^{2k + 2}) - \lambda(p^{2k}) }{p^{(k + 1)(1 + \frac{2 \pi i \tau}{L})}} d\tau + O(X^{1/2}\log\log X).
\end{align}
\end{lem}

\begin{proof} Recall $S_{\rm even, 2}$ is
\be
S_{\rm even,2} \ = \  \frac{1}{L X^*} \sum_{d\in\mathcal{F}( X)}\sum_{k = 1}^\infty \sum_{p|d} \frac{(\alpha_p^{2k} + \beta_p^{2k}) \log p}{p^k} \widehat{g}\left(\frac{\log p^k}{L} \right),
\ee
and a change of order of summation gives
\be \label{sixtythree}
S_{\rm even,2} \ = \  \frac{1}{L X^*} \sum_{p}\sum_{k = 1}^\infty \frac{(\alpha_p^{2k} + \beta_p^{2k}) \log p}{p^k} \widehat{g}\left(\frac{\log p^k}{L} \right) \sum_{\substack{d\in\mathcal{F}( X) \\ p|d}} 1.
\ee
From Lemma \ref{lem:numdwithpdivided} we find that
\be \label{NumberFundamental}
\twocase{\sum_{\substack{d\in\mathcal{F}(X) \\ p|d}} 1 \ = \
}{\frac{X^*}{p + 1} + O(X^{1/2})}{if $p \nmid  M$}{0}{if $p | M$.}
\ee

Using \eqref{NumberFundamental} in \eqref{sixtythree} yields
\be \label{Seventwohelpone}
S_{\rm even,2} \ = \  \frac{1}{L} \sum_{p\nmid M}\sum_{k = 1}^\infty \frac{(\alpha_p^{2k} + \beta_p^{2k}) \log p}{p^k(p + 1)} \widehat{g}\left(\frac{\log p^k}{L} \right)+ O(X^{1/2}\log\log X).
\ee

Substituting \be
\widehat{g} \left(\frac{\log p^k}{L} \right)\ =\ \int_{-\infty}^\infty g(\tau)e^{-2\pi i\tau \frac{\log p^k}{L}}d\tau
\ = \  \int_{-\infty}^\infty g(\tau)p^{-\frac{2\pi i \tau}{L}k}d\tau
\ee
into \eqref{Seventwohelpone} yields
\begin{align} \nonumber
S_{\rm even,2} & \ = \  \frac{1}{L} \sum_{p\nmid M}\sum_{k = 1}^\infty \frac{(\alpha_p^{2k} + \beta_p^{2k}) \log p}{p^k(p + 1)} \int_{-\infty}^\infty g(\tau)p^{-\frac{2\pi i \tau}{L}k}d\tau + O(X^{1/2}\log\log X) \\ \nonumber
& \ = \  \frac{1}{L} \sum_{p\nmid M}\sum_{k = 1}^\infty \frac{(\alpha_p^{2k} + \beta_p^{2k}) \log p}{p^k(p + 1)} \int_{-\infty}^\infty g(\tau)p^{-\frac{2\pi i \tau}{L}k}d\tau + O(X^{1/2}\log\log X) \\
& \ = \  \frac{1}{L} \int_{-\infty}^\infty g(\tau) \sum_{p\nmid M} \frac{\log p}{(p + 1)} \sum_{k = 1}^\infty \frac{(\alpha_p^{2k} + \beta_p^{2k}) }{p^{k(1 + \frac{2 \pi i \tau}{L})}} d\tau + O(X^{1/2}\log\log X).
\end{align}

For $p \nmid M$ we have
\be
\alpha_p^{2k} + \beta_p^{2k} \ = \  \lambda(p^{2k}) - \lambda(p^{2k - 2}),
\ee
thus
\begin{align} \nonumber
S_{\rm even,2} & \ = \  \frac{1}{L} \int_{-\infty}^\infty g(\tau) \sum_{p\nmid M} \frac{\log p}{(p + 1)} \sum_{k = 1}^\infty \frac{\lambda(p^{2k}) - \lambda(p^{2k - 2}) }{p^{k(1 + \frac{2 \pi i \tau}{L})}} d\tau + O(X^{1/2}\log\log X) \\
& \ = \  \frac{1}{L} \int_{-\infty}^\infty g(\tau) \sum_{p\nmid M} \frac{\log p}{(p + 1)} \sum_{k = 0}^\infty \frac{\lambda(p^{2k + 2}) - \lambda(p^{2k}) }{p^{(k + 1)(1 + \frac{2 \pi i \tau}{L})}} d\tau + O(X^{1/2}\log\log X).
\end{align}
\end{proof}


\subsection{Analysis of $S_{\rm odd}$}\label{sec:analysisSoddJutila}

We now analyze $S_{\rm odd}$ by applying Theorem \ref{thm:generalizedJutila}, which generalizes Jutila's bound. In the sums below, $M$ is an odd prime and $d$ is an even fundamental discriminant congruent to a non-zero square modulo $M$. We modify the analysis of $S_{\rm odd}$ from \cite{Mil4}, where the $S_{{\rm odd}}$ term is now \be S_{\rm odd}\ =\ -\frac{1}{X^*} \sum_{d\in\mathcal{F}( X)} \sum_{k = 0}^\infty \sum_p \frac{(\alpha_p^{2k + 1} + \beta_p^{2k +1})\chi_d(p) \log p}{p^{(2k+1)/2} L} \widehat{g}\left(\frac{\log p^{2k+1}}{2L} \right),
\ee with the $d$-sum over fundamental discriminants such that $d$ equals a non-zero square modulo $M$. If $p \nmid M$ then $\alpha_p^{2k+1} + \beta_p^{2k+1} = \lambda_E(p^{2k+1}) - \lambda_E(p^{2k-1})$, provided we set $\lambda_E(p^{-1}) = 0$; if $p|M$ then $\beta_p = 0$, $\alpha_p = \lambda_E(p)$ and therefore $\alpha_p^{2k+1} = \lambda_E(p)^{2k+1}$. Thus we may re-write our sum as \bea S_{\rm odd}& \ =\ & -\frac{1}{X^*}  \sum_{k = 0}^\infty \sum_{p\nmid M} \frac{(\lambda_E(p^{2k+1}) - \lambda_E(p^{2k-1})) \log p}{p^{(2k+1)/2} L} \widehat{g}\left(\frac{\log p^{2k+1}}{2L} \right) \sum_{d\in\mathcal{F}( X) \atop d \equiv \Box \neq 0 \bmod M} \chi_d(p) \nonumber\\ & & \ -\frac{1}{X^*}  \sum_{k = 0}^\infty \sum_{p|M} \frac{\lambda_E(p)^{2k+1} \log p}{p^{(2k+1)/2} L} \widehat{g}\left(\frac{\log p^{2k+1}}{2L} \right) \sum_{d\in\mathcal{F}( X) \atop d \equiv \Box \neq 0 \bmod M} \chi_d(p).
\eea

\begin{lem}  We have \bea\label{eq:Soddwithlowerorderterm}\label{Soddfinal}  S_{\rm odd} = -\frac1{L} \int_{-\infty}^\infty g(\tau)
\Bigg[
\sum_{k=0}^\infty  \frac{\log M}{M^{\frac{2k+1}{2}(2 + 2\frac{\pi i \tau}{L})}}
\Bigg] d \tau + \ O_M\left(X^{-\frac{1-\sigma}2}  \log^6 X\right). \eea
\end{lem}

\begin{proof} We write $S_{\rm odd}$ as $S_{\rm odd}(p\nmid M) + S_{\rm odd}(p|M)$. We first analyze $S_{\rm odd}(p|M)$, the contribution from $M$. As $d = \Box \not\equiv 0 \bmod M$, $\chi_d(M) = \lag{d}{M} = 1$. The $d$-sum is just $X^\ast$, and hence these terms contribute \be -\sum_{k=0}^\infty  \frac{\lambda_E(M)^{2k+1} \log M}{M^{(2k+1)/2} L} \widehat{g}\left(\frac{\log M^{2k+1}}{2L} \right). \ee

We apply Cauchy-Schwartz to $S_{\rm odd}(p \nmid M)$, and from Theorem \ref{thm:generalizedJutila} (our generalization of Jutila's bound) find \bea\label{eq:soddapplyJ} \left|S_{{\rm odd}}(p\nmid M)\right| &\ \le \ &
\frac1{X^\ast}\left(\sum_{\ell=0}^\infty \sum_{p^{2\ell+1} \le
X^\sigma \atop p \nmid M} \left|\frac{\log p}{p^{(2\ell+1)/2} \log X}\
\hg\left(\frac{\log p^{2\ell+1}}{\log
X}\right)\right|^2\right)^{1/2} \nonumber\\ & & \ \ \cdot \
\left(\sum_{\ell=0}^\infty \sum_{p^{2\ell+1} \le X^\sigma \atop (p,M) = 1}
\left|\sum_{d\le X \atop d \equiv \Box \neq 0 \bmod M} \chi_d(p)\right|^2\right)^{1/2} \nonumber\\ &\ll
& \frac1{X^\ast}\left(\sum_{n \le X^\sigma} \frac1{n}\right)^{1/2}
\cdot X^{\frac{1+\sigma}2} \log^5 X \nonumber\\ &\ll &
X^{-\frac{1-\sigma}2}  \log^6 X; \eea thus there is a power savings
if $\sigma < 1$.

We substitute for $\widehat{g}((\log M^{2k+1})/2L)$ its expansion as an integral, and find\begin{equation}
S_{\rm odd} = -\frac1{L} \int_{-\infty}^\infty g(\tau)
\Bigg[
\sum_{k=0}^\infty  \frac{\lambda_E(M)^{2k+1} \log M}{M^{\frac{2k+1}{2}(1 + 2\frac{\pi i \tau}{L})}}
\Bigg] d \tau + \ O_M\left(X^{-\frac{1-\sigma}2}  \log^6 X\right).
\end{equation}

For $p | M$ we have
\begin{equation}
\lambda_E(p)\ =\ \omega_E/p^{1/2} \Rightarrow \lambda_E(M)^{2k+1}\ =\ \frac{\omega_E}{M^{\frac{2k+1}2}} \ = \ \frac{1}{M^{\frac{2k+1}2}}
\end{equation}
since our elliptic curve $E$ has even functional equation. Thus
\begin{equation}
S_{\rm odd} = -\frac1{L} \int_{-\infty}^\infty g(\tau)
\Bigg[
\sum_{k=0}^\infty  \frac{\log M}{M^{\frac{2k+1}{2}(2 + 2\frac{\pi i \tau}{L})}}
\Bigg] d \tau + \ O_M\left(X^{-\frac{1-\sigma}2}  \log^6 X\right).
\end{equation} \end{proof}

\subsection{Proof of Theorem \ref{thm:numbertheoryresult}}

\begin{proof}[Proof of Theorem \ref{thm:numbertheoryresult}] The proof of \eqref{eq:numbtheoryresulteq} follows by collecting the above lemmata and
noticing that from equation (\ref{Sevenoneonefinal}) for $S_{\rm even,1,1}$ and
equation (\ref{Soddfinal}) for $S_{\rm odd}$ we have
\begin{align}
S_{\rm even,1,1} + S_{\rm odd} & = -\frac{1}{L} \sum_{k=1}^\infty \int_{-\infty}^\infty g(\tau) \frac{\log M}{M^{2k(1 + \frac{\pi i \tau}{L})}} d \tau \nonumber \\ \nonumber
&\ \ \ \ \  -\ \frac1{L} \int_{-\infty}^\infty g(\tau)
\Bigg[
\sum_{k=0}^\infty  \frac{\log M}{M^{\frac{2k+1}{2}(2 + 2\frac{\pi i \tau}{L})}}
\Bigg] d \tau + O_M\left(X^{-\frac{1-\sigma}2}  \log^6 X\right)\\
& = -\frac{1}{L} \sum_{k=0}^\infty \int_{-\infty}^\infty g(\tau) \frac{\log M}{M^{(k+1)(1 + \frac{\pi i \tau}{L})}} d \tau
+ O_M\left(X^{-\frac{1-\sigma}2}  \log^6 X\right).
\end{align}
\end{proof}


\section{The Ratios Conjecture's Prediction} \label{derivativeofAE}

The purpose of this section is to prove Theorem \ref{thm:ntequalsratios}, specifically that if $\supp(\hphi) \subset (-\sigma, \sigma)$ then the Ratios' prediction agrees with number theory up to errors of size $O(X^{-(1-\sigma)/2})$. The starting point in the analysis is the following expansion for the Ratios Conjecture's prediction:

\begin{thm}[Theorem 2.3 and equation (3.11) in \cite{HKS}]
With notation as in Theorem \ref{thm:numbertheoryresult}, the prediction from the Ratios Conjecture for the
one-level density of the family ${\mathcal F}(X)$ of even quadratic twists of an elliptic curve L-function $L_E(s)$
of even functional equation by even fundamental discriminants at most $X$ is
\begin{eqnarray}
&& \frac{1}{X^*}\sum_{d \in {\mathcal F}(X)} \sum_{\gamma_d}g\Big(\frac{\gamma_d
L}{\pi}\Big)\nonumber
\\ \nonumber && =~ \frac{1}{2LX^*} \int_{-\infty}^\infty g(\tau) \sum_{d \in {\mathcal F}(X)} \Bigg[ 2\log
\left(\frac{\sqrt{M}|d|}{2\pi} \right) +
\frac{\Gamma'}{\Gamma}\Big(1 + \frac{i\pi \tau}{L}\Big)
\\ \nonumber &&~~~ + \frac{\Gamma'}{\Gamma}\Big(1 - \frac{i \pi \tau}{L}\Big) +
2\Big[-\frac{\zeta'(1 + \frac{2 i \pi \tau}{L})}{\zeta(1 + \frac{2
i \pi \tau}{L})} + \frac{L_E'(\sym^2, 1 + \frac{2 i \pi
\tau}{L})}{L_E(\sym^2, 1+\frac{2 i \pi \tau}{L})} +
A_E^1\Big(\frac{i \pi \tau}{L}, \frac{i \pi \tau}{L}\Big)
\\ \nonumber &&~~~ -\bigg(\frac{\sqrt{M}|d|}{2\pi}\bigg)^{-2i \pi \tau / L}
\frac{\Gamma(1 - \frac{i \pi \tau}{L})}{\Gamma(1 + \frac{i \pi
\tau}{L})} \frac{\zeta(1 + \frac{2 i \pi \tau}{L})L_E(\sym^2, 1 -
\frac{2 i \pi \tau}{L})}{L_E(\sym^2,1)}
 \times A_E\Big(-\frac{i \pi
\tau}{L}, \frac{i \pi \tau}{L}\Big)\Big] \Bigg]d\tau \\ &&~~~ + O(X^{-1/2+\varepsilon}).
\end{eqnarray}
where $A_E$ is defined in \eqref{definitionofAE} and $\dalpha A_E(\alpha, \gamma)|_{\alpha = \gamma = r} = A_E^1(r,r)$.
\end{thm}

Much of the expansion above is already found in our number theory result, Theorem \ref{thm:numbertheoryresult}. The proof of Theorem \ref{thm:ntequalsratios} is thus reduced to determining the contribution from the $A_E$ and $A_E^1$ terms, which we now proceed to do in the lemmata below. We first derive useful expressions for these pieces and the related quantities that arise in the analysis. Similar to \cite{Mil4}, the proof is completed by bounding the contribution of the resulting Euler product by shifting contours.

\subsection{Analysis of $A_E^1$}\label{sec:analysisAE1}

Before determining the contribution of $A_E^1$ we first obtain a useful expansion for it. The Euler product $A_E(\alpha, \gamma)$ is given by
\begin{eqnarray}\label{definitionofAE}
& & A_E(\alpha, \gamma) \nonumber\\
& = & Y_E^{-1}(\alpha, \gamma)\times
\prod_{p|M}\Bigg(\sum_{m= 0}^{\infty}\bigg( \frac{\lambda(p^m)\omega_E^{m}}
{p^{m(1/2 + \alpha)
}}-\frac{\lambda(p)}{p^{1/2+\gamma}}\frac{\lambda(p^m)\omega_E^{m+1}}{p^{m(1/2+\alpha)}}\bigg)\Bigg) \times
\nonumber\\ & &
\prod_{p\nmid M} \left(1 + \frac{p}{p + 1}\left(\sum_{m =
1}^\infty \frac{\lambda(p^{2m})}{p^{m(1 + 2\alpha)}}
-\frac{\lambda(p)}{p^{1 + \alpha + \gamma}}\sum_{m =
0}^\infty \frac{\lambda(p^{2m+1})}{p^{m(1 + 2\alpha)}} +
\frac{1}{p^{1 + 2\gamma}}\sum_{m = 0}^\infty
\frac{\lambda(p^{2m})}{p^{m(1 + 2\alpha)}} \right) \right)
\end{eqnarray}
where \begin{equation}\label{eq:YE}
 Y_E(\alpha, \gamma) = \frac{\zeta(1 +
2\gamma) L_E(\sym^2, 1 + 2\alpha)}{\zeta(1 + \alpha + \gamma)
L_E(\sym^2, 1+ \alpha+\gamma)}.
\end{equation}
Note that
\be
A_E(r,r) = 1.
\ee
Rewriting $A_E(\alpha, \gamma)$ gives
\begin{align*}
A_E(\alpha, \gamma) = & \prod_{p | M}
\left(1 - \frac{1}{p^{1 + 2\gamma}}\right)
\left(1 - \frac{\lambda(p)^2}{p^{1 + 2 \alpha}} \right)
\left(1 - \frac{1}{p^{1 + \alpha + \gamma}} \right)^{-1}
\left(1 - \frac{\lambda(p)^2}{p^{1 + \alpha + \gamma}}\right)^{-1}\\
& \times \Bigg(\sum_{m= 0}^{\infty}\bigg( \frac{\lambda(p^m)\omega_E^{m}}
{p^{m(1/2 + \alpha)
}}-\frac{\lambda(p)}{p^{1/2+\gamma}}\frac{\lambda(p^m)\omega_E^{m+1}}{p^{m(1/2+\alpha)}}\bigg)\Bigg)\\
& \prod_{p \nmid M}
\left(1 - \frac{1}{p^{1 + 2\gamma}}\right)
\left(1 - \frac{\lambda(p^2)}{p^{1 + 2\alpha}} + \frac{\lambda(p^2)}{p^{2(1+2\alpha)}} -\frac{1}{p^{3(1+ 2\alpha)}} \right)
\left(1 - \frac{1}{p^{1 + \alpha + \gamma}} \right)^{-1}\\
& \times \left(1 - \frac{\lambda(p^2)}{p^{1 + \alpha + \gamma}} + \frac{\lambda(p^2)}{p^{2(1 + \alpha + \gamma)}} - \frac{1}{p^{3(1 + \alpha + \gamma)}} \right)^{-1}\\
& \times
\left(1 + \frac{p}{p + 1}
\left(\sum_{m = 1}^\infty \frac{\lambda(p^{2m})}{p^{m(1 + 2\alpha)}}
- \frac{\lambda(p)}{p^{1 + \alpha + \gamma}}\sum_{m = 0}^\infty \frac{\lambda(p^{2m+1})}{p^{m(1 + 2\alpha)}}
 + \frac{1}{p^{1 + 2\gamma}}\sum_{m = 0}^\infty \frac{\lambda(p^{2m})}{p^{m(1 + 2\alpha)}}
\right)
\right).\nonumber\\
\end{align*}

We find
\begin{align*}
& \dalpha A_E(\alpha, \gamma)  \\= & A_E(\alpha, \gamma)
\Bigg(
\sum_{p | M} \log p
\Bigg[
\frac{\frac{2 \lambda(p)^2}{p^{1 + 2\alpha}}}{1 - \frac{\lambda(p)^2}{p^{1+2\alpha}}}
- \frac{\frac{1}{p^{1 + \alpha + \gamma}}}{1 - \frac{1}{p^{1 + \alpha + \gamma}} }
- \frac{\frac{\lambda(p)^2}{p^{1+ \alpha + \gamma}}}{1 - \frac{\lambda(p)^2}{p^{1 + \alpha + \gamma}}}\\
&\ \
+ \frac{- \sum_{m= 0}^{\infty}\bigg( \frac{m\lambda(p^m)\omega_E^{m}}
{p^{m(1/2 + \alpha)
}}-\frac{m\lambda(p)}{p^{1/2+\gamma}}\frac{\lambda(p^m)\omega_E^{m+1}}{p^{m(1/2+\alpha)}}\bigg)}
{\sum_{m= 0}^{\infty}\bigg( \frac{\lambda(p^m)\omega_E^{m}}
{p^{m(1/2 + \alpha)
}}-\frac{\lambda(p)}{p^{1/2+\gamma}}\frac{\lambda(p^m)\omega_E^{m+1}}{p^{m(1/2+\alpha)}}\bigg)}
\Bigg]\\
&\ \ \sum_{p\nmid M}
\log p
\Bigg[
\frac{\frac{2\lambda(p^2)}{p^{1 + 2\alpha}} - \frac{4\lambda(p^2)}{p^{2(1+2\alpha)}} + \frac{6}{p^{3(1 + 2\alpha)}}}
{1 - \frac{\lambda(p^2)}{p^{1 + 2\alpha}} + \frac{\lambda(p^2)}{p^{2(1+2\alpha)}} - \frac{1}{p^{3(1 + 2\alpha)}}}
- \frac{\frac{1}{p^{1 + \alpha + \gamma}}}{1 - \frac{1}{p^{1 + \alpha + \gamma}} }\\
&\ \ +
\frac{-\frac{\lambda(p^2)}{p^{1 + \alpha + \gamma}} + \frac{2\lambda(p^2)}{p^{2(1+\alpha + \gamma)}} - \frac{3}{p^{3(1 + \alpha + \gamma)}}}
{
1 - \frac{\lambda(p^2)}{p^{1 + \alpha + \gamma}} + \frac{\lambda(p^2)}{p^{2(1+\alpha + \gamma)}} - \frac{1}{p^{3(1 + \alpha + \gamma)}}
}\\
&\ \ + \frac{
\frac{p}{p + 1}\left(- \sum_{m = 1}^\infty \frac{2m\lambda(p^{2m})}{p^{m(1 + 2\alpha)}} + \frac{\lambda(p)}{p^{1 + \alpha + \gamma}}\sum_{m = 0}^\infty \frac{(2m+1)\lambda(p^{2m+1})}{p^{m(1 + 2\alpha)}} - \frac{1}{p^{1 + 2\gamma}}\sum_{m = 0}^\infty \frac{2m\lambda(p^{2m})}{p^{m(1 + 2\alpha)}}
\right)
}
{
\left(1 + \frac{p}{p + 1}\left(\sum_{m = 1}^\infty \frac{\lambda(p^{2m})}{p^{m(1 + 2\alpha)}}
- \frac{\lambda(p)}{p^{1 + \alpha + \gamma}}\sum_{m = 0}^\infty \frac{\lambda(p^{2m+1})}{p^{m(1 + 2\alpha)}} + \frac{1}{p^{1 + 2\gamma}}\sum_{m = 0}^\infty \frac{\lambda(p^{2m})}{p^{m(1 + 2\alpha)}}
\right) \right)
}
\Bigg]
\Bigg).
\end{align*}

Specializing to $\alpha = \gamma = r$ we find that
\begin{align} \nonumber
& \dalpha A_E(\alpha, \gamma)|_{\alpha = \gamma = r} = A_E^1(r,r) \\ \nonumber
& = \sum_{p | M} \log p
\Bigg[
\frac{\frac{2\lambda(p)^2}{p^{1 + 2r}}}{1 - \frac{\lambda(p)^2}{p^{1 + 2r}}} - \frac{\frac{\lambda(p)^2}{p^{1 + 2r}}}{1 - \frac{\lambda(p)^2}{p^{1 + 2r}}}
- \frac{\frac{1}{p^{1 + 2r}}}{1 - \frac{1}{p^{1 + 2r}}}
- \sum_{m = 0}^\infty \frac{\lambda(p^{m+1})\omega_E^{m + 1}}{p^{(m + 1)(1/2 + r)}}
\Bigg] \\ \nonumber
&\ \
+ \sum_{p\nmid M} \log p
\Bigg[
\frac{\frac{2\lambda(p^2)}{p^{1+2r}} - \frac{4\lambda(p^2)}{p^{2(1+2r)}} + \frac{6}{p^{3(1+2r)}}}{1 - \frac{\lambda(p^2)}{p^{1+2r}} + \frac{\lambda(p^2)}{p^{2(1+2r)}} - \frac{1}{p^{3(1+2r)}}}
+ \frac{-\frac{\lambda(p^2)}{p^{1+2r}} + \frac{2\lambda(p^2)}{p^{2(1+2r)}} - \frac{3}{p^{3(1+2r)}}}{1 - \frac{\lambda(p^2)}{p^{1+2r}} + \frac{\lambda(p^2)}{p^{2(1+2r)}} - \frac{1}{p^{3(1+2r)}}}
- \frac{\frac{1}{p^{1 + 2r}}}{1 - \frac{1}{p^{1 + 2r}}}\\ \label{Aonederivative}
&\ \
-\sum_{m = 0}^\infty \frac{\lambda(p^{2m + 2}) - \lambda(p^{2m})}{p^{(m + 1)(1+2r)}}
+\frac{1}{p + 1}
\sum_{m = 0}^\infty \frac{\lambda(p^{2m + 2}) - \lambda(p^{2m})}{p^{(m + 1)(1+2r)}}
\Bigg].
\end{align}
Next, we identity terms in (\ref{Aonederivative}) involving the logarithmic derivatives of $\zeta(s)$ and $L_E(\sym^2,s)$. Simple calculations show
\be \label{derivativelogzeta}
\frac{\zeta'(1 + 2r)}{\zeta(1 + 2r)} = - \sum_p \log p \frac{\frac{1}{p^{1 + 2r}}}{1 - \frac{1}{p^{1 + 2r}}}
\ee
and
\be \label{derivativelogLsym}
\frac{L_E'(\sym^2, 1 + 2r)}{L_E(\sym^2, 1 + 2r)}
= -\sum_{p | M} \log p \frac{\frac{\lambda(p)^2}{p^{1 + 2r}}}{1 - \frac{\lambda(p)^2}{p^{1 + 2r}}}
- \sum_{p\nmid M} \log p \frac{\frac{\lambda(p^2)}{p^{1 + 2r}} - \frac{2 \lambda(p^2)}{p^{2(1+2r)}} + \frac{3}{p^{3(1+2r)}}}{1 - \frac{\lambda(p^2)}{p^{1 + 2r}} + \frac{\lambda(p^2)}{p^{2(1 + 2r)}} - \frac{1}{p^{3(1 + 2r)}}}.
\ee

Also note that
\be
\frac{\zeta'(1 + 2r)}{\zeta(1 + 2r)}\ =\ -\frac{\tilde{\zeta}'(1 + 2r)}{\tilde{\zeta}(1 + 2r)}
\ee
where
\be
\tilde{\zeta}(s)\ =\ \zeta^{-1}(s);
\ee
similarly we have
\be
\frac{L_E'(\sym^2, 1 + 2r)}{L_E(\sym^2, 1 + 2r)} = -\frac{\tilde{L}_E'(\sym^2, 1 + 2r)}{\tilde{L}_E(\sym^2, 1 + 2r)}
\ee
where
\be
\tilde{L}_E(\sym^2, 1 + 2r)\ =\ L_E^{-1}(\sym^2, 1 + 2r).
\ee

Using \eqref{derivativelogzeta} and \eqref{derivativelogLsym} in \eqref{Aonederivative} yields
\begin{eqnarray}
A_E^1(r,r) & = & -2\frac{L_E'(\sym^2, 1 + 2r)}{L_E(\sym^2, 1 + 2r)} + \frac{L_E'(\sym^2, 1 + 2r)}{L_E(\sym^2, 1 + 2r)}  + \frac{\zeta'(1 + 2r)}{\zeta(1 + 2r)} \nonumber\\
& & \ \ \ \ -\sum_{p | M} \log p \sum_{m = 0}^\infty \frac{\lambda(p^{m+1})\omega_E^{m + 1}}{p^{(m + 1)(1/2 + r)}} \nonumber\\
& &\ \ \ \ +\sum_{p\nmid M} \log p
\Big[
-\sum_{m = 0}^\infty \frac{\lambda(p^{2m + 2}) - \lambda(p^{2m})}{p^{(m + 1)(1+2r)}}
+\frac{1}{p + 1}
\sum_{m = 0}^\infty \frac{\lambda(p^{2m + 2}) - \lambda(p^{2m})}{p^{(m + 1)(1+2r)}}
\Big].\nonumber
\end{eqnarray}
Hence
\begin{align} \nonumber
A_E^1(r,r) & = -\frac{L_E'(\sym^2, 1 + 2r)}{L_E(\sym^2, 1 + 2r)}  + \frac{\zeta'(1 + 2r)}{\zeta(1 + 2r)}
 -\sum_{p | M} \log p \sum_{m = 0}^\infty \frac{\lambda(p^{m+1})\omega_E^{m + 1}}{p^{(m + 1)(1/2 + r)}} \\
 \label{AoneTermUmgeschrieben}
&\ \ +\sum_{p\nmid M} \log p
\Big[
-\sum_{m = 0}^\infty \frac{\lambda(p^{2m + 2}) - \lambda(p^{2m})}{p^{(m + 1)(1+2r)}}
+\frac{1}{p + 1}
\sum_{m = 0}^\infty \frac{\lambda(p^{2m + 2}) - \lambda(p^{2m})}{p^{(m + 1)(1+2r)}}
\Big].
\end{align}

\kommentar
{
Notice that for $p\nmid M$ we have
\be \label{helpone}
\alpha_p^{2k + 2} + \beta_p^{2k + 2} = \lambda(p^{2k+ 2}) - \lambda(p^{2k}),
\ee
for $p|M$ we have
\be \label{helptwo}
\lambda(p) = p^{-1/2},
\ee
and
\be \label{helpthree}
\frac{L_E'(\sym^2, s)}{L_E(\sym^2, s)} = -\sum_{p\nmid M} \log p \sum_{k = 1}^\infty \frac{\alpha_p^{2k} + \beta_p^{2k} + 1}{p^{sk}} - \sum_{p | M} \log p \sum_{k = 1}^\infty \frac{\alpha_p^{2k} + \beta_p^{2k}}{p^{sk}}.
\ee
}

\begin{lem}[Contribution of $A_E^1$]\label{lem:contrAE1} We have
\begin{align} \nonumber
& \frac{1}{LX^*} \int_{-\infty}^\infty g(\tau) \sum_{d \in \mathcal{F}(X)} A_E^1\left(\frac{i \pi \tau}{L},\frac{i \pi \tau}{L}\right) d\tau \\
& = \frac{1}{L} \int_{-\infty}^\infty g(\tau)\left( -\sum_{p | M} \log p \sum_{m = 0}^\infty \frac{1}{p^{(m + 1)(1 + r)}}\right.\nonumber\\
& \,\,\,\,\,\,\,\,\,\,\,\,\,\,\,\,\,\,\,\,\,\,\,\,\,\,\,\,\,\,\,\,\,\,\,\,\,\,\,\,\left.+ \sum_{p\nmid M}  \frac{\log p}{p + 1}
\sum_{k = 0}^\infty \frac{\lambda(p^{2k + 2}) - \lambda(p^{2k})}{p^{(k + 1)(1+\frac{2i \pi \tau}{L})}} -\sum_{\ell=1}^\infty\frac{(M^\ell-1)\log M}{M^{(2r+2)\ell}}\right)d\tau.\label{finalAEoneInPrediction}
\end{align}
\end{lem}

\begin{proof} The sign $\varepsilon_f$ of a modular form $f$ of weight $k$ and level $M$ is (see equation (3.5) of \cite{ILS})
\be
\varepsilon_f = i^k \mu(M) \lambda(M) \sqrt{M}.
\ee
In our case we denote $\varepsilon_f$ with $\omega_E$. As $k$ is 2 and $M$ is a prime, $i^k = i^2 = -1$ and $\mu(M) = -1$, so
\be
\omega_E = (-1)(-1) \lambda(M) \sqrt{M} \Rightarrow \lambda(M) = \frac{\omega_E}{\sqrt{M}}.
\ee
In particular we obtain for $p|M$ that
\be \label{sign_functionalequation}
\lambda(p^{m + 1}) \omega_E^{m + 1}\ =\ \left(\frac{\omega_E}{p^{1/2}}\right)^{m + 1} \omega_E^{m +1 }\ =\ p^{-(m+1)/2},
\ee
and for $p|M$ we have
\be
\lambda(p) = \frac{\omega_E}{p^{1/2}}.
\ee Hence in \eqref{AoneTermUmgeschrieben} we have
\be
-\sum_{p | M} \log p \sum_{m = 0}^\infty \frac{\lambda(p^{m+1})\omega_E^{m + 1}}{p^{(m + 1)(1/2 + r)}}
= -\sum_{p | M} \log p \sum_{m = 0}^\infty \frac{1}{p^{(m + 1)(1 + r)}}.
\ee

Collecting terms, we find
\begin{eqnarray}\label{almostfinalAEone}
A_E^1(r,r) & = & -\frac{L_E'(\sym^2, 1 + 2r)}{L_E(\sym^2, 1 + 2r)}  + \frac{\zeta'(1 + 2r)}{\zeta(1 + 2r)} \nonumber\\
& & \ \ \ \ -\sum_{p | M} \log p \sum_{m = 0}^\infty \frac{1}{p^{(m + 1)(1 + r)}}
- \sum_{p\nmid M} \log p \sum_{m = 0}^\infty \frac{\lambda(p^{2m + 2}) - \lambda(p^{2m})}{p^{(m + 1)(1+2r)}}\nonumber\\
& &\ \ \ \  + \sum_{p\nmid M}  \frac{\log p}{p + 1}
\sum_{m = 0}^\infty \frac{\lambda(p^{2m + 2}) - \lambda(p^{2m})}{p^{(m + 1)(1+2r)}}
\nonumber\\&=&-\sum_{p | M} \log p \sum_{m = 0}^\infty \frac{1}{p^{(m + 1)(1 + r)}}+ \sum_{p\nmid M}  \frac{\log p}{p + 1}
\sum_{m = 0}^\infty \frac{\lambda(p^{2m + 2}) - \lambda(p^{2m})}{p^{(m + 1)(1+2r)}}+B(r,r),\nonumber\\
\end{eqnarray}
where $B(r,r)$ is the sum of the first pair of terms and the fourth term.  Expanding the logarithmic derivatives\footnote{If $\RRe(r) > 0$ the series converge and the cancelation is justified; the result holds for all $r$ by analytic continuation.} (see Equation \eqref{logderivativesym}, etc.) and using the identity $\lambda(p^{2m})-\lambda(p^{2m-2})=\alpha_p^{2m}+\beta_p^{2m}$, we have
\begin{align}\nonumber
B(r,r) & =  -\frac{L_E'(\sym^2, 1 + 2r)}{L_E(\sym^2, 1 + 2r)}
+ \frac{\zeta'(1 + 2r)}{\zeta(1 + 2r)}
- \sum_{p\nmid M} \log p \sum_{m = 0}^\infty \frac{\lambda(p^{2m + 2}) - \lambda(p^{2m})}{p^{(m + 1)(1+2r)}}\nonumber\\
& = \sum_{p\nmid M}\sum_{\ell=1}^\infty \frac{(\alpha_p^{2\ell} +\beta_p^{2\ell})\log p}{p^{(1+2r)\ell}}+\sum_{p\nmid M}\sum_{\ell=1}^\infty\frac{\log p}{p^{(1+2r)\ell}}+\sum_{p|M}\sum_{\ell=1}^\infty\frac{\log p}{p^{((1+2r)+1)\ell}}\nonumber\\
& \,\,\,\,\,\,\,\,\,\,-\sum_{p}\sum_{\ell=1}^\infty\frac{\log p}{p^{(1+2r)\ell}}- \sum_{p\nmid M} \log p \sum_{m = 1}^\infty \frac{\alpha_p^{2k}+\beta_p^{2k}}{p^{m(1+2r)}}\nonumber\\
& = \sum_{p\nmid M}\log p\sum_{\ell=1}^\infty\frac{\alpha_p^{2\ell} +\beta_p^{2\ell}-\alpha_p^{2\ell} -\beta_p^{2\ell}+1-1}{p^{(1+2r)\ell}}\nonumber\\
& \,\,\,\,\,\,\,\,\,\,-\sum_{p\nmid M}\sum_{\ell=1}^\infty\frac{\log p}{p^{(1+2r)\ell}}+\sum_{p|M}\sum_{\ell=1}^\infty\frac{\log p}{p^{((1+2r)+1)\ell}}\nonumber
\\=&-\sum_{\ell=1}^\infty\frac{(M^\ell-1)\log M}{M^{(2r+2)\ell}}.
\end{align}
This calculation implies that
\begin{align}
A_E^1(r,r) & =  -\sum_{p | M} \log p \sum_{m = 0}^\infty \frac{1}{p^{(m + 1)(1 + r)}}+ \sum_{p\nmid M}  \frac{\log p}{p + 1}
\sum_{m = 0}^\infty \frac{\lambda(p^{2m + 2}) - \lambda(p^{2m})}{p^{(m + 1)(1+2r)}}\nonumber\\
& \,\,\,\,\,\,\,\,\,\,-\sum_{\ell=1}^\infty\frac{(M^\ell-1)\log M}{M^{(2r+2)\ell}} \label{finalAEone}.
\end{align}

We are concerned with the term
\be
\frac{1}{LX^*} \int_{-\infty}^\infty g(\tau) \sum_{d \in \mathcal{F}(X)} A_E^1\left(\frac{i \pi \tau}{L},\frac{i \pi \tau}{L}\right) d\tau \ee from the Ratios' prediction. Using \eqref{finalAEone} yields \eqref{finalAEoneInPrediction}, completing the proof.
\end{proof}

\subsection{Analysis of $A_E$}

Recapping our analysis to date, we have shown the Ratios' prediction is
\begin{eqnarray} \nonumber
&& \frac{1}{X^*}\sum_{d \in \mathcal{F}(X)} \sum_{\gamma_d}g\Big(\frac{\gamma_d
L}{\pi}\Big)
\\ \nonumber
&& =~ \frac{1}{2LX^*} \int_{-\infty}^\infty g(\tau) \sum_{d \in \mathcal{F}(X)}
\Bigg[
2\log
\left(\frac{\sqrt{M}|d|}{2\pi} \right) +
\frac{\Gamma'}{\Gamma}\Big(1 + \frac{i\pi \tau}{L}
\Big) + \frac{\Gamma'}{\Gamma}\Big(1 - \frac{i \pi \tau}{L}\Big)
\Bigg]
d\tau
\\ \nonumber
&& + \frac{1}{L}\intii
g(\tau) \left(-\frac{\zeta'}{\zeta}\left(1+\frac{2\pi i \tau}{L}\right)+ \frac{L_E'}{L_E}\left({\rm sym}^2,1+\frac{2\pi i \tau}{L}\right)-\sum_{\ell=1}^\infty\frac{(M^\ell-1)\log M}{M^{\left(2+\frac{2i\pi\tau}{L}\right)\ell}}\right) d\tau\\ \label{formularfinal}
&& ~~~-\frac{1}{L} \sum_{k=0}^\infty \int_{-\infty}^\infty g(\tau) \frac{\log M}{M^{(k+1)(1 + \frac{\pi i \tau}{L})}} d \tau+\frac{1}{L} \int_{-\infty}^\infty g(\tau) \sum_{p\nmid M} \frac{\log p}{(p + 1)} \sum_{k = 0}^\infty \frac{\lambda(p^{2k + 2}) - \lambda(p^{2k}) }{p^{(k + 1)(1 + \frac{2 \pi i \tau}{L})}} d\tau\nonumber\\
&& -\frac{1}{LX^*} \int_{-\infty}^\infty g(\tau) \sum_{d \in \mathcal{F}(X)}
\Bigg[
\bigg(\frac{\sqrt{M}|d|}{2\pi}\bigg)^{-2i \pi \tau / L}
\frac{\Gamma(1 - \frac{i \pi \tau}{L})}{\Gamma(1 + \frac{i \pi
\tau}{L})} \frac{\zeta(1 + \frac{2 i \pi \tau}{L})L_E(\sym^2, 1 -
\frac{2 i \pi \tau}{L})}{L_E(\sym^2,1)}\nonumber\\
&& \times A_E\Big(-\frac{i \pi
\tau}{L}, \frac{i \pi \tau}{L}\Big)
\Bigg]d\tau  \label{predictionfinal}
 + O(X^{-1/2+\varepsilon}).
\end{eqnarray}
Comparing \eqref{predictionfinal} and the one-level density from number theory (Theorem \ref{thm:numbertheoryresult}),
we see that we have agreement in all but two terms -- first, the constant $g(0)/2$; second, a term from (\ref{predictionfinal}) requiring analysis, namely
\begin{align} \label{AEtermtoanalyse}
&-\frac{1}{LX^*} \int_{-\infty}^\infty g(\tau)
\sum_{d \in \mathcal{F}(X)}
\Bigg[
\bigg(\frac{\sqrt{M}|d|}{2\pi}\bigg)^{-2i \pi \tau / L}
\frac{\Gamma(1 - \frac{i \pi \tau}{L})}{\Gamma(1 + \frac{i \pi
\tau}{L})} \frac{\zeta(1 + \frac{2 i \pi \tau}{L})L_E(\sym^2, 1 -
\frac{2 i \pi \tau}{L})}{L_E(\sym^2,1)}\\ \nonumber
& \times A_E\Big(-\frac{i \pi
\tau}{L}, \frac{i \pi \tau}{L}\Big)
\Bigg]d\tau.
\end{align}

The proof of Theorem \ref{thm:ntequalsratios} is thus reduced to proving

\begin{lem}\label{lem:expansionAEterm} The contribution from the $A_E$ term to the Ratios' prediction, given by \eqref{AEtermtoanalyse}, equals $g(0)/2$ plus an error term bounded by $O(X^{-\frac{1-\sigma}2})$. \end{lem}

Before proving Lemma \ref{lem:expansionAEterm} we first derive a useful expansion. We consider the following term from (\ref{AEtermtoanalyse}):
\begin{equation}
T(\tau) \ :=\ \frac{\zeta(1 + \frac{2 i \pi \tau}{L})L_E(\sym^2, 1 -
\frac{2 i \pi \tau}{L})}{L_E(\sym^2,1)} \times A_E\Big(-\frac{i \pi
\tau}{L}, \frac{i \pi \tau}{L}\Big).
\end{equation}
Our goal is to replace this with a uniformly convergent Euler product times $\zeta\left(1 + 2{i\pi\tau}/{L}\right)$, with the residue at $\tau=0$ readily computable. We let $s > 1$ be a free parameter. From the expansion of $A_E(\alpha, \gamma)$ in (\ref{definitionofAE}) we have
\begin{equation} \label{Tterm}
T(\tau) \ =\ \left((\zeta(s) \times V_{\nmid}\Big(-\frac{i \pi
\tau}{L}s, \frac{i \pi \tau}{L}s\Big) \times V_{|}\Big(-\frac{i \pi
\tau}{L}s, \frac{i \pi \tau}{L}s\Big)\right)\Bigg|_{s=1}
\end{equation}
where (see \cite{HKS}, equations (2.17) and (2.18)) we introduced the following to improve convergence:
\begin{eqnarray} \label{Euler}
V_\nmid(\alpha, \gamma) &= &\prod_{p\nmid M} \left(1 + \frac{p}{p
+ 1}\left(\sum_{m = 1}^\infty \frac{\lambda(p^{2m})}{p^{m(1 +
2\alpha)}} - \frac{\lambda(p)}{p^{1 + \alpha + \gamma}}\sum_{m =
0}^\infty \frac{\lambda(p^{2m+1})}{p^{m(1 +
2\alpha)}}+ \frac{1}{p^{1 + 2\gamma}}\sum_{m
= 0}^\infty \frac{\lambda(p^{2m})}{p^{m(1 + 2\alpha)}} \right)
\right)\nonumber\\
V_{|}(\alpha,\gamma)&=&\prod_{p| M}\Bigg(\sum_{m= 0}^{\infty}\bigg(
\frac{\lambda(p^m)\omega_E^{m}} {p^{m(1/2 + \alpha)
}}-\frac{\lambda(p)\lambda(p^m)\omega_E^{m+1}}{p^{m(1/2+\alpha)+1/2+\gamma}}\bigg)\Bigg).
\end{eqnarray}
From \cite{HKS}, equation (2.31) we have
\begin{eqnarray}
V_{\nmid}(\alpha, \gamma)
&=& \label{orderforeulerproduct} \prod_{p\nmid M} \left(1 +
\frac{\lambda(p^2)}{p^{1 + 2\alpha}} - \frac{\lambda(p^2) + 1
}{p^{1 + \alpha + \gamma}} + \frac{1}{p^{1 + 2 \gamma}} +
\cdots\right),
\end{eqnarray}
where the $\cdots$ indicate terms that converge like $1/p^2$ when
$\alpha$ and $\gamma$ are small.

In (\ref{Tterm}) the contribution from the lone bad prime $M$ is readily managed, and does not affect the convergence or divergence of the product. We are left with
\begin{eqnarray}
\tilde{T}(\tau) & := & \left(\zeta(s) \times V_{\nmid}\Big(-\frac{i \pi \tau}{L}s, \frac{i \pi \tau}{L}s\Big)\right)\Bigg|_{s=1}\nonumber\\
& = & \left(\prod_p \left(1 +
\frac{\lambda(p^2)}{p^{1 - 2\frac{i \pi \tau}{L}s}} - \frac{\lambda(p^2) + 1
}{p} + \frac{1}{p^{1 + 2 \frac{i \pi \tau}{L}s}} +
\cdots\right) \left(1 + \frac{1}{p^s}+\cdots\right)\right)\Bigg|_{s=1} \nonumber\\
& =& \left(\prod_p \left(1 +
\frac{\lambda(p^2)}{p^{1 - 2\frac{i \pi \tau}{L}s}} - \frac{\lambda(p^2) + 1
}{p} + \frac{1}{p^{1 + 2 \frac{i \pi \tau}{L}s}} +
\frac{1}{p^s}\right.\right.\nonumber\\ & & \ \ \ \ \ +\ \left.\left.\frac{\lambda(p^2)}{p^{1+s - 2\frac{i \pi \tau}{L}s}}
- \frac{\lambda(p^2) + 1
}{p^{1+s}} + \frac{1}{p^{1+s + 2 \frac{i \pi \tau}{L}s}} +
\cdots\right)\right)\Bigg|_{s=1}\nonumber\\
& = & \left( \prod_p \left(1 +
\frac{\lambda(p^2)}{p^{1 - 2\frac{i \pi \tau}{L}s}} - \frac{\lambda(p^2)
}{p} + \frac{1}{p^{1 + 2 \frac{i \pi \tau}{L}s}} -\frac{1}{p}\left(1-\frac{1}{p^{s-1}}\right)+
\cdots\right)\right)\Bigg|_{s=1}.\nonumber\\
\end{eqnarray}

Note that the $(1/p)\left(1-{1}/{p^{s-1}}\right)$ term goes to $0$ as $s\rightarrow 1$.  Also note that (cf. \cite{HKS}, (2.32) and (2.33))
\begin{equation}
L_E(\sym^2, 1 - 2i \pi \tau/L) = \prod_p \left(1 + \frac{\lambda(p^2)}{p^{1 - 2\frac{i \pi \tau}{L}}} + \cdots\right),
\end{equation}
and
\begin{equation}
\frac{1}{L_E(\sym^2, 1)}\ =\ \prod_p \left(1 - \frac{\lambda(p^2)}{p} + \cdots \right), \ \ \
\zeta\left(1 + 2 \frac{i \pi \tau}{L}\right)\ =\ \prod_p\left(1 + \frac{1}{p^{1 + 2 \frac{i \pi \tau}{L}}} +
\cdots\right).
\end{equation}

Thus
\begin{equation}
T(\tau)\ = \ K(\tau) \times \frac{L_E(\sym^2, 1 - 2i \pi \tau/L)}{L_E(\sym^2, 1)} \times \zeta\left(1 + 2 \frac{i \pi \tau}{L}\right)
\end{equation}
where $K(\tau)$ is a convergent Euler product that converges uniformly in the region of interest and equals 1 when $\tau = 0$ (the last claim follows from analyzing our above expansion at $\tau=0$ and comparing with the expressions in \S\ref{sec:analysisAE1}). In particular, we know that $K(\tau)=\prod_p\left(1+O(1/p^2)\right)$; if there were any higher order terms, we would have a term of higher order that $1/p^2$ in the expansion of $\tilde{T}(\tau)$ besides those already accounted for, which does not occur.

\begin{proof}[Proof of Lemma \ref{lem:expansionAEterm}]
Instead of analyzing (\ref{AEtermtoanalyse}), it suffices to show
\begin{align} \nonumber
R(g, X) = &-\frac{1}{LX^*} \int_{-\infty}^\infty g(\tau)
\sum_{d \in \mathcal{F}(X)}
\Bigg[
\bigg(\frac{\sqrt{M}|d|}{2\pi}\bigg)^{-2i \pi \tau / L}
\frac{\Gamma(1 - \frac{i \pi \tau}{L})}{\Gamma(1 + \frac{i \pi
\tau}{L})} \\ \nonumber
& \ \ \ \ \times\
 \frac{L_E(\sym^2, 1 - 2i \pi \tau/L)}{L_E(\sym^2, 1)}
\times K(\tau) \times \zeta\left(1 + 2 \frac{i \pi \tau}{L}\right)\Bigg]
d\tau
\end{align}
is ${g(0)}/{2}+O(X^{-\frac{1-\sigma}{2}})$.
Recall from \eqref{eq:scalingdefinition}
that \begin{equation} L\ =\  \log\left(\frac{\sqrt{M}X}{2\pi} \right).
\end{equation}
By Lemma \ref{lem:partialsumdexpdpiX}
 \be \sum_{d\in\mathcal{F}(X)} \left(\frac{\sqrt{M}d}{2\pi}\right)^{-\frac{2\pi i \tau}{L}} \ = \
X^\ast e^{-2\pi i \tau} \left(1 - \frac{2\pi i \tau}{L}\right)^{-1} + O(X^{1/2}\log X). \ee The $O(X^{1/2})$ term yields a contribution of size $O(X^{-1/2})$, which is negligible. Thus it suffices to study the main term, which we denote $R_1(g,X)$.

We replace $\tau$ with $\tau - iw\frac{L}{2\pi}$ with $w=0$ (we
will shift the contour in a moment).
Thus \bea R_1(g;X) & \ = \ &
-\frac{X^\ast}{LX^*} \intii g\left(\tau -
iw\frac{L}{2\pi}\right)
 e^{-2\pi i \left(\tau - iw\frac{L}{2\pi}\right) } \frac{\Gamma(1 - \frac{w}{2} - \frac{i \pi \tau}{L})}{\Gamma(1 + \frac{w}{2} + \frac{i \pi
\tau}{L})}
\nonumber\\ & & \ \ \ \ \ \ \cdot \  \frac{L_E(\sym^2, 1 - w - 2i \pi \tau/L)}{L_E(\sym^2, 1)} \cdot K(\tau)
\cdot \zeta\left(1 + w + 2 \frac{i \pi \tau}{L}\right)\Bigg]
d\tau.\eea We now shift the contour to $w = 3/2$. Remembering we are
assuming the GRH for $\zeta(s)$ and $L_E(\sym^2,\rho)$ (so that if $\zeta(\rho) = 0$ or $L_E(\sym^2,s) = 0$ then
either $\rho = \foh + i\gamma$ for some $\gamma \in \R$ or $\rho$ is
a negative even integer),  there are two
different residue contributions as we shift, arising from

\bi \item the pole of $\zeta\left(1 + w + \frac{2\pi i \tau}{L}\right)$ at $w=\tau=0$; \item the zeros of $L_E\left(\sym^2, 1 - w - 2i \pi \tau/L\right)$ when $w=1/2$ and $\tau = \gamma
\frac{L}{2\pi}$. \ei\

We claim the contribution from the pole of $\zeta\left(\sym^2,1 + w +
\frac{2\pi i \tau}{L}\right)$ at $w=\tau=0$ is $g(0)/2$. As the pole of $\zeta(s)$
is $1/(s-1)$, since $s = 1 + \frac{2\pi i \tau}{L}$ the
$1/\tau$ term from the zeta function has coefficient $\frac{L}{2\pi i}$. We lose the factor of $1/2\pi i$ when we apply the
residue theorem, there is a minus sign outside the integral and
another from the direction we integrate (we replace the integral
from $-\gep$ to $\gep$ with a semi-circle oriented clockwise; this
gives us a minus sign as well as a factor of $1/2$ since we only
have half the contour), and everything else evaluated at $\tau = 0$
is $g(0)$ (remember $K(0)=1$).


We now analyze the contribution from the zeros of $L_E(\sym^2,s)$ as we
shift $w$ to $3/2$. The contributions from the non-trivial zeros arise when $w = 1/2$, and we sum over $\tau = \gamma
\frac{L}{2\pi}$ with $L_E(\sym^2,\foh+i\gamma) = 0$. The $\exp\left(-2\pi i (\tau - iw\frac{L}{2\pi})\right)$ term is $O(\exp(-L/2)) = O(X^{-1/2})$, and the $K$-piece is bounded as it is uniformly convergent in this region.

From (3) of Lemma \ref{lem:usefulfacts} we have \be g\left(\gamma
\frac{L}{2\pi}-i\frac12\frac{L}{2\pi}\right) \ \ll \
X^{\sigma/2} (\tau^2+1)^{-B}\ee for any $B > 0$. From (4) of Lemma
\ref{lem:usefulfacts}, we see that the ratio of the Gamma factors is
bounded by a power of $|\tau|$. Finally, the zeta function in the numerator is
$O(1)$. Thus the contribution from the critical zeros
of $L_E(\sym^2,s)$ is bounded by \be \sum_{\gamma \atop
L_E({\text{sym}}^2,\foh+i\gamma) = 0} \ X^{-1/2} X^{\sigma/2} \int \frac{d\tau}{(\tau^2+1)^B} \ \ll \ X^{-\frac{1-\sigma}2} \ee for sufficiently large $B$. Thus there is a power savings in this term so long as $\sigma < 1$; note, however, that we \textbf{\emph{do not}} obtain square-root cancellation in this error term for \textbf{\emph{any}} support. This is very different than \cite{Mil4}, and is due to the different ratio of $L$-functions arising in this case, leading to a more complicated Euler product.

The proof is completed by a standard argument showing that the integral over $w=3/2$ is
negligible. Arguing as above shows the integral is bounded by
$O(X^{-3/2+3\sigma/2})$. It suffices to obtain polynomial in $\tau$ bounds for $L_E(\sym^2, -1/2 - 2\pi i \tau/L)$; see for instance \cite{IK}. This completes the proof of Lemma \ref{lem:expansionAEterm}, which also finishes the proof of Theorem \ref{thm:ntequalsratios}.
\end{proof}

\begin{rek} We sketch an alternate start of the proof of the above lemma. One difficulty is that $R_1(g;X)$ is defined as an
integral and there is a pole on the line of integration. We may
write \be \zeta(s) \ = \ (s-1)^{-1} \ +\ \left(\zeta(s) -
(s-1)^{-1}\right). \ee For us $s = 1 + \frac{2\pi i \tau}{L}$,
so the first factor is just $\frac{L}{2\pi i \tau}$. As
$g(\tau)$ is an even function, the main term of the integral of this
piece is \bea \intii g(\tau) \frac{e^{-2\pi i \tau}}{2\pi i \tau} \
d\tau & \ = \ & \intii g(\tau) \left(\frac{e^{-2\pi i \tau}}{4\pi i
\tau} - \frac{e^{2\pi i \tau}}{4\pi i \tau}\right) d\tau \nonumber\\
& = & -\intii g(\tau) \frac{\sin(2\pi \tau)}{2\pi \tau}\ d\tau \ = \
-\frac{g(0)}2, \eea where the last equality is a consequence of
$\supp(\hg) \subset (-1,1)$. The other terms from the $(s-1)^{-1}$
factor and the terms from the $\zeta(s) - (s-1)^{-1}$ piece are
analyzed in a similar manner as the terms in the proof of Lemma
\ref{lem:expansionAEterm}.
\end{rek}

\begin{rek} The proof of Lemma \ref{lem:expansionAEterm} follows from
shifting contours and keeping track of poles of ratios of Gamma,
zeta and $L$-functions. Arguing as in Remark 2.3 of \cite{Mil3} we can prove a related result with significantly
less work, specifically, agreement up to any power of the logarithm.
\end{rek}


\section{Generalizing Jutila's bound}\label{sec:jutila}

In these notes we generalize Jutila's bound, and show how it may be applied to analyze the contribution from odd powers of primes to the 1-level density of families of quadratic twists of a fixed ${\rm GL}_n$ form. While we are most interested in the case when the fixed form is an elliptic curve of prime conductor, we prove our bound in greater generality as this may be of use to other researchers. In particular, this result was implicitly assumed by Rubinstein \cite{Rub} in his analysis of the main term in the 1-level density of quadratic twists of a fixed form.

Recall Jutila's bound (see (3.4) of \cite{Ju3}) is \be \sum_{1 < n \le N
\atop n\ {\rm non-square}} \ \left| \sum_{0 < d \le X \atop d\ {\rm
fund.\ disc.}}\ \chi_d(n)\right|^2 \ \ll \ N X \log^{10} N, \ee where
the $d$-sum is over even fundamental discriminants at most $X$. For many applications we need to modify it further. Let $M$ be a square-free integer. We often need to restrict the $d$-sum to be over $d$ relatively prime to $M$ that are congruent to a non-zero square modulo $M$. We have $\chi_d(n) = \lag{d}{n}$, where $\lag{d}{n}$ is the Kronecker symbol. We can encode the restriction on the $d$-sum by noting \be \twocase{\foh\left(\chi_d(M)^2 + \chi_d(M)\right) \ = \ }{1}{if $d$ is a non-zero square modulo $M$ and $(d,M) = 1$}{0}{otherwise;} \ee if instead we wanted to detect $d$ a non-square modulo $M$ we would use $\chi_d(M)^2 - \chi_d(M)$.

\begin{thm}[Generalization of Jutila's bound]\label{thm:generalizedJutila} Let $M$ be a square-free positive integer. Then \be \sum_{1 < n \le N, (n,M) = 1
\atop n\ {\rm non-square}}\left( \sum_{d \le X, (d,M) = 1 \atop d \equiv \Box \neq 0 \bmod M} \chi_d(n)\right)^2 \ \ll \ NM^2X\log^{10}(NM). \ee The same bound holds if instead we restrict the $d$-sum to be over non-squares modulo $M$. \end{thm}

\begin{proof}
In all sums below, $d$ and $d'$ denote an even fundamental discriminant. Letting $S(N,M,X)$ denote our sum of interest, we find \bea S(N,M,X) & \ = \ & \sum_{1 < n \le N, (n,M) = 1
\atop n\ {\rm non-square}}\left( \sum_{d \le X, (d,M) = 1 \atop d \equiv \Box \neq 0 \bmod M} \chi_d(n)\right)^2 \nonumber\\ &=& \fof \sum_{1 < n \le N, (n,M) = 1 \atop n\ {\rm non-square}} \left( \sum_{d \le X} \chi_d(n) \chi_d(M)^2 + \sum_{d \le X} \chi_d(n)\chi_d(M)\right)^2 \nonumber\\ &=& S_1(N,M,X) + S_2(N,M,X) \eea (using the estimate $(a+b)^2 \le 4a^2 + 4b^2$), where \bea S_1(N,M,X) & \ = \ & \sum_{1 < n \le N, (n,M) = 1 \atop n\ {\rm non-square}} \left( \sum_{d \le X} \chi_d(n) \chi_d(M)^2\right)^2 \nonumber\\ S_2(N,M,X) & \ = \ & \sum_{1 < n \le N, (n,M) = 1 \atop n\ {\rm non-square}} \left( \sum_{d \le X} \chi_d(n) \chi_d(M)\right)^2. \eea

The first sum, $S_1(N,M,x)$, is easily estimated using Jutila's bound. Note that $\chi_d(n)\chi_d(M^2)$ $=$ $\chi_d(nM^2)$, and if $n$ is not a square at most $N$ then $nM^2$ is not a square at most $NM^2$. Thus \be S_1(N,M,X) \ \ll \ NM^2 X \log^{10} (NM^2)  \ \ll \ NM^2 X \log^{10} (NM) \ee (while Jutila's bound is over all square-free $n$, as it is a sum of squares we can restrict the sum over $n$). The second sum is handled similarly, using $\chi_d(n)\chi_d(M) = \chi_d(nM)$. As $M$ is prime and $(n,M) = 1$, $nM$ is not a square at most $NM$. Thus \be  S_2(N,M,X) \ \ll \ NM X \log^{10} (NM). \ee We therefore find \be S(N,M,X) \ \ll \ NM^2X\log^{10} (NM). \ee
\end{proof}

\begin{rek} Not surprisingly, we restrict to $n$ relatively prime to $M$ in Theorem \ref{thm:generalizedJutila}; if $n=M$ then since $d \equiv \Box \neq 0 \bmod d$, $\chi_d(n)$ would equal 1 and these terms would contribute on the order of $X^2$ to the sum. \end{rek}

\begin{rek}
Rubinstein \cite{Rub} calculated the main term in the 1-level density for the family of quadratic twists of a fixed form on ${\rm GL}_n$, where the fundamental discriminants used in twisting were additionally restricted so that the family had constant sign. In his work he implicitly assumed that Jutila's bound (which was the key arithmetic ingredient in the number theory calculations of the 1-level density for the family of quadratic characters) still held when the fundamental discriminants were further restricted as above; Theorem \ref{thm:generalizedJutila} justifies this assumption, and almost suffices to complete the analysis. Unlike our present work, where we are attempting to determine all lower order terms up to square-root cancelation, in \cite{Rub} the goal is just to show agreement between the main term and the predictions from random matrix theory. Thus we do not need to identify the term corresponding to the $1/L$ term from \eqref{eq:Soddwithlowerorderterm}. We thus simply follow the argument in \cite{Rub} and trivially bound the contribution from primes dividing $M$ (which we now assume is just square-free and not necessarily prime).
\end{rek}




\appendix

\setcounter{equation}{0}
\section{Explicit Formula}\label{sec:expformula}

We fix an elliptic curve $E$ with prime conductor $M$ and let $L_E(s)$ be the $L$-function attached to $E$.
We denote the quadratic twists of $L_E(s)$ by $L_E(s,\chi_d)$. For $(d, M ) = 1$ the completed $L$-function of $L_E(s,\chi_d)$ is
\begin{equation} \label{CompletedLfunction}
\Lambdachid = \left(\frac{2\pi}{\sqrt{M}|d|}\right)^{-s - 1/2} \Gamma(s + 1/2) L_E(s,\chi_d)
\end{equation}
which relates $s$ to $1-s$, i.e.,
\begin{equation} \label{srelatingoneminuss}
\Lambdachid = \chi_d(-M)\omega_E \Lambda(1-s, \chi_d).
\end{equation} As we are only interested in the quadratic twists with even functional equation
we have
\begin{equation} \label{signoffunctionalequation}
\chi_d(-M)\omega_E = +1.
\end{equation} Taking the logarithmic derivative of \eqref{CompletedLfunction} gives
\begin{equation} \label{logarithmicderivativeofLambda}
\Lambdaderivative = \log\left( \frac{\sqrt{M}|d|}{2\pi} \right)
+ \frac{\Gamma'(s+ 1/2)}{\Gamma(s + 1/2)} + \frac{L_E'(s,\chi_d)}{L_E(s,\chi_d)}.
\end{equation}

For later use we express the logarithmic derivative of $L_E(s,\chi_d)$ as a sum over primes.
We note that
\begin{align} \nonumber
\Lchid & = \prod_{p | M}\left(1 - \frac{\lambda(p)\chi_d(p)}{p^s}\right)^{-1}\prod_{p\nmid M}\left(1 - \frac{\lambda(p)\chi_d(p)}{p^s} + \frac{\chi_d^2(p)}{p^{2s}}\right)^{-1}\\ \label{LEulerProduct}
& = \prod_p \left(1 - \frac{\alpha_p \chi_d(p)}{p^s}\right)^{-1} \left(1 - \frac{\beta_p \chi_d(p)}{p^s}\right)^{-1}
\end{align}
where the above product is over all primes,
\begin{equation}
\alpha_p + \beta_p = \lambda(p)
\end{equation}
and
\begin{equation}
\alpha_p \beta_p =
\begin{cases}
0 \mbox{~if~} p|M\\
1 \mbox{~if~} p\nmid M.
\end{cases}
\end{equation}
The logarithmic derivative of \eqref{LEulerProduct} is
\begin{equation} \label{logarithmeticderivativeofL}
\frac{L_E'(s,\chi_d)}{L_E(s,\chi_d)}
= -\sum_p \log p \sum_{k =1}^\infty \frac{(\alpha_p^k + \beta_p^k)\chi_d^k(p)}{p^{sk}}.
\end{equation}
We assume GRH, so if $\tfrac{1}{2} + i \gamma$ denotes a zero of $\Lambda(s, \chi_d)$ we have
$\gamma \in \RR$. Let $\phi$ denote an even Schwartz function where its Fourier transform
\begin{equation}
\widehat{\phi}(\xi) = \int_{-\infty}^\infty \phi(x) e^{-2\pi i x \xi} d x
\end{equation}
has finite support, i.e., $\supp(\widehat{\phi}) \subset (-\sigma, \sigma)$ for some finite $\sigma$.
We extend $\phi(x)$ to the whole complex plane via
\begin{equation}
H(s) = \phi\left(\frac{s - \tfrac{1}{2}}{i}\right).
\end{equation}

The starting point of all one-level density investigations is the explicit formula; the derivation below is modified from \cite{Mes,RS}.

\begin{lem} The one-level density for the family of quadratic twists by even fundamental discriminants of a fixed elliptic curve $E$ with even functional equation and prime conductor $M$ is
\begin{align} \nonumber
& \frac{1}{X^*}\sum_{d \in \mathcal{F}(X)}\sum_{\gamma_d} g\left(\gamma_d \frac{L}{\pi}\right) \\ \nonumber
& = \frac{1}{2 LX^*} \int_{-\infty}^\infty g(\tau) \sum_{d \in \mathcal{F}(X)}
\left[2\log\left(\frac{\sqrt{M}|d|}{2 \pi} \right)
+ \frac{\Gamma'}{\Gamma}\left(1 + i\frac{\pi \tau}{L}\right) +\frac{\Gamma'}{\Gamma}\left(1 - i\frac{\pi \tau}{L}\right) \right] d \tau\\
& \ \ \ \ - \ \frac{2}{2 L} \sum_{d \in \mathcal{F}(X)} \sum_{k = 1}^\infty \sum_p \frac{(\alpha_p^k + \beta_p^k)\chi_d^k(p)\log p}{p^{k/2}}
\widehat{g}\left(\frac{\log p^k}{2 L} \right),
\end{align}
where ${\mathcal{F}(X)}$ denotes the family of interest,
$$
{\mathcal{F}(X)} = \left\{0 < d \leq X : d {\rm\ an\ even\ fundamental\ discriminant\ and\ } \chi_d(-M)\omega_E = 1 \right\},
$$ and \be X^\ast\ =\ |\mathcal{F}(X)|, \ \ \ \ \ L\ =\ \log\left(\frac{\sqrt{M}X}{2\pi} \right). \ee
\end{lem}

\begin{proof}
We set
\begin{equation} \label{definitionofI}
I = \frac{1}{2\pi i } \int_{\RRe(s) = 3/2} \Lambdaderivative H(s) ds.
\end{equation}

We shift the contour to $\RRe(s) = -1/2$. The only contribution is from the zeros of
$\Lambdachid$. Hence we obtain
\begin{equation} \label{Ione}
I = \sum_\gamma \phi(\gamma) + \frac{1}{2 \pi i} \int_{\RRe(s) = -1/2} \Lambdaderivative H(s) ds.
\end{equation}

By \eqref{srelatingoneminuss} and \eqref{signoffunctionalequation} we have
\begin{equation} \label{Lambdaone}
\Lambdachid = \Lambdachidone
\end{equation}
and therefore also
\begin{equation} \label{Lambdatwo}
\Lambda'(s,\chi_d) = - \Lambda'(1-s,\chi_d).
\end{equation}

With \eqref{Lambdaone} and \eqref{Lambdatwo} in \eqref{Ione} we obtain
\begin{equation}
I = \sum_\gamma \phi(\gamma) - \frac{1}{2 \pi i} \int_{\RRe(s) = -1/2} \frac{\Lambda'(1-s,\chi_d)}{\Lambda(1-s,\chi_d)}  H(s) ds.
\end{equation}
A change of variable $s \rightarrow 1 - s$ yields
\begin{equation} \label{Lambdathree}
I = \sum_\gamma \phi(\gamma) - \frac{1}{2 \pi i} \int_{\RRe(s) = 3/2} \Lambdaderivative  H(1-s) ds.
\end{equation}

Combining \eqref{definitionofI} and \eqref{Lambdathree} gives
\begin{equation}
\sum_\gamma \phi(\gamma) = \frac{1}{2 \pi i} \int_{\RRe(s) = 3/2} \Lambdaderivative [H(s) + H(1-s)]ds.
\end{equation}
Using \eqref{logarithmicderivativeofLambda} we expand the logarithmic derivative of $\Lambdachid$ and shift
the contours of all terms except the $L'(s,\chi_d)/L(s,\chi_d)$ term to $\RRe(s) = 1/2$. (Recall that $H(s)$
is even and symmetric about $s = \tfrac{1}{2}$.)
The result is
\begin{equation} \label{Lambdafour}
\sum_\gamma \phi(\gamma) = I_1 + I_2
\end{equation}
where
\begin{equation}
I_1 = \frac{1}{2 \pi i} \int_{\RRe(s) = 1/2} \left[\log\left(\frac{\sqrt{M}|d|}{2 \pi} \right) + \frac{\Gamma'}{\Gamma}(s + 1/2)\right] [H(s) + H(1-s)]ds \label{Ieins}
\end{equation}
and
\begin{equation}
I_2 = \frac{1}{2 \pi i} \int_{\RRe(s) = 3/2} \frac{L_E'(s, \chi_d)}{L_E(s, \chi_d)} [H(s) + H(1-s)]ds.
\end{equation}
The integral in \eqref{Ieins} with $s = \tfrac{1}{2} + iy$ is
\begin{align} \nonumber
I_1 = & \frac{1}{2 \pi i} \int_{-\infty}^\infty \left[\log\left(\frac{\sqrt{M}|d|}{2 \pi} \right)
+ \frac{\Gamma'}{\Gamma}(\tfrac{1}{2} + iy + \tfrac{1}{2})\right] 2 \phi(y) i dy\\ \nonumber
& = \frac{1}{2 \pi} \int_{-\infty}^\infty \left[2\log\left(\frac{\sqrt{M}|d|}{2 \pi} \right)
+ 2 \frac{\Gamma'}{\Gamma}(1 + iy)\right]\phi(y)dy \\ \label{endresultIone}
& = \frac{1}{2 \pi} \int_{-\infty}^\infty \left[2\log\left(\frac{\sqrt{M}|d|}{2 \pi} \right)
+ \frac{\Gamma'}{\Gamma}(1 + iy) +\frac{\Gamma'}{\Gamma}(1 - iy) \right]\phi(y)dy.
\end{align}
Now we analyze the integral $I_2$ in \eqref{Lambdafour}, which is
\begin{align}
I_2 & = \frac{1}{2 \pi i} \int_{\RRe(s) = 3/2} \frac{L_E'(s, \chi_d)}{L_E(s, \chi_d)} [H(s) + H(1-s)]ds.
\end{align}
We shift the contour to $\RRe(s) = 1/2$ and use \eqref{logarithmeticderivativeofL} to obtain
\begin{equation}
I_2 = -\frac{1}{2 \pi i} \sum_{k = 1}^\infty \sum_p \log p (\alpha_p^k + \beta_p^k) \chi_d^k(p)
\int_{\RRe(s) = 1/2} [H(s) + H(1-s)]e^{-ks\log p} ds.
\end{equation}
A change of variable $s = \tfrac{1}{2} + iy$ yields
\begin{align} \nonumber
I_2 & = -\frac{2}{2 \pi} \sum_{k = 1}^\infty \sum_p (\alpha_p^k + \beta_p^k)\chi_d^k(p) \log p
\int_{-\infty}^\infty \phi(y) e^{-k(1/2 +iy)\log p} dy\\ \nonumber
& = -\frac{2}{2 \pi} \sum_{k = 1}^\infty \sum_p \frac{(\alpha_p^k + \beta_p^k)\chi_d^k(p)\log p}{p^{k/2}}
\int_{-\infty}^\infty \phi(y) e^{-2\pi i y \frac{\log p^k}{2\pi}} d y\\ \label{endresultItwo}
& = -\frac{2}{2 \pi} \sum_{k = 1}^\infty \sum_p \frac{(\alpha_p^k + \beta_p^k)\chi_d^k(p)\log p}{p^{k/2}}
\widehat{\phi}\left(\frac{\log p^k}{2\pi} \right).
\end{align}

Thus with \eqref{endresultIone} and \eqref{endresultItwo} we obtain the following explicit formula for
the one-level density:
\begin{align} \nonumber
\sum_\gamma \phi(\gamma)
& = \frac{1}{2 \pi} \int_{-\infty}^\infty \left[2\log\left(\frac{\sqrt{M}|d|}{2 \pi} \right)
+ \frac{\Gamma'}{\Gamma}(1 + iy) +\frac{\Gamma'}{\Gamma}(1 - iy) \right]\phi(y)dy\\ \label{explicitformulaone}
& ~~~-\frac{2}{2 \pi} \sum_{k = 1}^\infty \sum_p \frac{(\alpha_p^k + \beta_p^k)\chi_d^k(p)\log p}{p^{k/2}}
\widehat{\phi}\left(\frac{\log p^k}{2\pi} \right).
\end{align}

We slightly rewrite \eqref{explicitformulaone} by summing over the twists $d$ and scale the zeros by the
mean density of zeros. First we note for $g(x) = \phi(A \cdot x), A \not=0$ the Fourier transforms are
related through
\begin{equation}
\widehat{g}(\xi) = \frac{\widehat{\phi}(\xi/A)}{A}\label{gphirelation}.
\end{equation}

We set
\begin{equation}\label{eq:defnL}
L = \log\left(\frac{\sqrt{M}X}{2\pi} \right)
\end{equation}
and replace $\phi(y)$ in \eqref{explicitformulaone} with
\begin{equation}
g(\tau) = \phi\left(y\right)
\end{equation}
where $\tau = yL/\pi$. Note that other papers often denote our $L$ by $2L$; we use this notation to match \cite{HKS}, who calculated much of the Ratios' prediction for this family. Finally summing over the quadratic twists yields the claim.
\end{proof}

\setcounter{equation}{0}

\section{Sums over fundamental discriminants}\label{sec:sumsoverfunddiscs}

We generalize the calculations in Appendix B of \cite{Mil4} to handle our family, which has the added restriction of requiring our even fundamental discriminants $d$ to be a non-zero square modulo a prime $M$. We can encode the restriction on the $d$-sum by noting \be \twocase{\foh\left(\chi_d(M)^2 + \chi_d(M)\right) \ = \ }{1}{if $d$ is a non-zero square modulo $M$ and $(d,M) = 1$}{0}{otherwise;} \ee if instead we wanted to detect $d$ a non-square modulo $M$ we would use $\chi_d(M)^2 - \chi_d(M)$.

\begin{lem}\label{lem:numdwithpdivided} Let $d$ denote an even
fundamental discriminant at most $X$, and set \be X^\ast\ =\ \sum_{d \le
X \atop d = \Box \not\equiv 0 \bmod M} 1\ee for an odd prime $M$. Then\footnote{We chose to write $X^\ast$ to facilitate comparison with the cardinality of the corresponding family from \cite{Mil4}, where we did not impose the constraint that $d$ equal a non-zero square modulo $M$.} \be X^\ast   \ = \ \frac{3}{\pi^2}X \cdot \frac{M}{2(M+1)} + O(X^{1/2}) \ee and
for $p \le X^{1/2}$ we have \be \twocase{\sum_{d \le X, p|d  \atop d = \Box \not\equiv 0 \bmod M} 1 \ = \ }{
\frac{X^\ast}{p+1} + O(X^{1/2})}{if $p\nmid M$}{0}{if $p|M$.} \ee
\end{lem}

\begin{proof} We first prove the claim for $X^\ast$, and then
indicate how to modify the proof when $p|d$. We could show this by
recognizing certain products as ratios of zeta functions or by using
a Tauberian theorem; instead we shall give a straightforward proof
suggested to us by Tim Browning (see also \cite{OS1}).

We first assume that $d \equiv 1 \bmod 4$, so we are considering
even fundamental discriminants $\{d \le X: d \equiv 1 \bmod 4,
\mu(d)^2=1, d = \Box \not\equiv 0 \bmod M\}$; it is trivial to modify the arguments below for $d$
such that $d/4 \equiv 2$ or $3$ modulo $4$ and $\mu(d/4)^2=1$. Let
$\chi_4(n)$ be the non-trivial character modulo 4: $\chi_4(2m) = 0$
and \be \twocase{\chi_4(n) \ = \ }{1}{if $n\equiv 1 \bmod 4$}{0}{if
$n \equiv 3 \bmod 4$.} \ee We have \bea\label{eq:startSXanalysis}
S(X) & \ = \ & \sum_{d\le X,\ d = \Box \not\equiv 0 \bmod M \atop \mu(d)^2=1,\ d\equiv 1 \bmod 4} 1
\nonumber\\ & \ = \ & \sum_{d \le
X \atop 2 \notdiv d} \mu(d)^2 \cdot \frac{1 + \chi_4(d)}2 \frac{\chi_d(M)^2+\chi_d(M)}{2}\nonumber\\
&=& \frac14 \sum_{d\le X \atop (2M,d)=1} \mu(d)^2 + \frac14 \sum_{d\le
X} \mu(d)^2 \left[\chi_4(d) \left(\chi_d(M)^2+\chi_d(M)\right) - \chi_4(d)^2\chi_d(M)\right]\nonumber\\ & \ = \ & S_1(X) + S_2(X).\eea By M\"obius
inversion \be \twocase{\sum_{m^2|d} \mu(m) \ = \ }{1}{if $d$ is
square-free}{0}{otherwise.} \ee Thus \bea S_1(X) & \ = \ & \frac14
\sum_{d \le X \atop (2M,d)=1} \sum_{m^2|d} \mu(m) \nonumber\\ &=&
\frac14 \sum_{m \le X^{1/2}\atop (2M,m)=1} \mu(m) \cdot \sum_{d\ \le\
X/m^2 \atop (2M,d)=1} 1 \nonumber\\ &=& \frac14 \sum_{m\le X^{1/2}
\atop (2M,m)=1} \mu(m)\left(\frac{X}{m^2} \frac{\phi(2M)}{2M} + O(1)\right)
\nonumber\\ &=& \frac{X}{8}\frac{M-1}{M} \sum_{m=1 \atop (2M,m) = 1}^\infty
\frac{\mu(m)}{m^2} + O(X^{1/2}) \nonumber\\ &=& \frac18 \frac{M-1}{M}
\frac{6}{\zeta(2)} \cdot \left(1 - \frac1{2^2}\right)^{-1} \left(1 - \frac1{M^2}\right)^{-1}\cdot X +
O(X^{1/2}) \nonumber\\ &=& \frac{1}{\pi^2} \frac{M}{M+1}X + O(X^{1/2}) \eea
(because we are missing the factors corresponding to $2$ and $M$ in
$1/\zeta(2)$ above). To make this comparable to the sum from \cite{Mil4} (where we did not have the condition that $d = \Box \not\equiv 0 \bmod M$) we may rewrite the above as \be S_1(X) \ = \ \frac{2}{\pi^2} X \cdot \frac{M}{2(M+1)}. \ee
Arguing in a similar manner shows $S_2(X) =
O(X^{1/2})$; this is due to the presence of a non-principal character in each of the three sums of modulus at most $8M$ (we use quadratic reciprocity to replace $\chi_d(M)$ with a character of conductor at most $8M$). For example, let $\chi$ denote any of the three non-principal characters in the expansion of $S_2(X)$. Such a term contributes  \be
 \frac14 \sum_{m \le X^{1/2}} \chi(m^2)\mu(m) \sum_{d\le
X/m^2} \chi(d) \ \ll \ X^{1/2} \ee (because we are summing
$\chi$ at consecutive integers, and thus this sum is at most $8M$).

A similar analysis shows that the number of even fundamental
discriminants $d\le X$ with $d/4 \equiv 2$ or $3$ modulo $4$ is
$\frac{1}{\pi^2}X \cdot \frac{M}{2(M+1)} + O(X^{1/2})$. Thus \be \sum_{d\le X, d = \Box \not\equiv 0 \bmod M \atop d\ {\rm an\
even\ fund.\ disc.}}1 \ = \ X^\ast \ = \
\frac{3}{\pi^2}X \frac{M}{2(M+1)}+O(X^{1/2}).\ee

We may trivially modify the above calculations to determine the
number of even fundamental discriminants $d\le X$ with $p|d$ for a
fixed prime $p$. We first assume $p\equiv 1 \bmod 4$. In
\eqref{eq:startSXanalysis} we replace $\mu(d)^2$ with $\mu(pd)^2$,
$d \le X$ with $d \le X/p$, $(2M,d)=1$ with $(2Mp,d)=1$. As $d$ and $p$
are now relatively prime (after this change of variables), $\mu(pd) = \mu(p)\mu(d)$ and the main term
becomes \bea S_{1;p}(X) & \ = \ &
\frac14 \sum_{d \le X/p \atop (2Mp,d)=1} \sum_{m^2|d} \mu(m) \nonumber\\
&=& \frac14 \sum_{m \le (X/p)^{1/2}\atop (2Mp,m)=1} \mu(m) \cdot
\sum_{d\ \le\ (X/p)/m^2 \atop (2Mp,d)=1} 1 \nonumber\\ &=& \frac14
\sum_{m\le (X/p)^{1/2} \atop
(2Mp,m)=1} \mu(m)\left(\frac{X/p}{m^2} \cdot \frac{\phi(2Mp)}{2Mp} + O(1)\right) \nonumber\\
&=& \frac{(p-1)(M-1)X}{8Mp^2} \sum_{m=1 \atop (2Mp, m)=1}^\infty
\frac{\mu(m)}{m^2} + O(X^{1/2}) \nonumber\\ &=& \frac18
\frac{6}{\zeta(2)} \cdot \left(1 - \frac1{2^2}\right)^{-1}
\left(1-\frac1{p^2}\right)^{-1} \left(1 - \frac1{M^2}\right)^{-1} \frac{(p-1)(M-1)X}{Mp^2}\nonumber\\ & & \ \ \ \ \ +\ O(X^{1/2}) \nonumber\\
&=& \frac{2X}{(p+1)\pi^2} \frac{M}{2(M+1)} + O(X^{1/2}) \ = \ \frac{2X^\ast/3}{p+1} + O(X^{1/2}), \eea and the cardinality of
this piece is reduced by $(p+1)^{-1}$ (note above we used $\#\{n \le
Y: (2p,n)=1\}$ $=$ $\frac{p-1}{2p}Y+O(1)$). A similar analysis as before shows that $S_{2;p}(X) = O(X^{1/2})$; the case of even fundamental discriminants $d$ with $d/4 \equiv 2$ or $3$ modulo $4$ follows analogously.

We need to trivially modify the above arguments if $p\equiv 3 \bmod
4$ (if $p = M$ these arguments are not applicable, although in this case the result is clearly zero as we are only considering $d = \Box \not\equiv 0 \bmod M$, and such $d$ are never divisible by $M$). If for instance we require $d\equiv 1 \bmod 4$ then instead of
using the factor $\mu(d)^2(1 + \chi_4(d))/2$ we use $\mu(pd)^2 (1-\chi_4(d))/2$, and the rest of the proof proceeds
similarly.

It is a completely different story if $p=2$. Note if $d\equiv 1
\bmod 4$ then 2 \emph{never} divides $d$, while if $d/4 \equiv 2$ or
3 modulo 4 then 2 \emph{always} divides $d$. There are $3X/\pi^2 \cdot \frac{M}{2(M+1)} +
o(X^{1/2})$ even fundamental discriminants at most $X$, and $X/\pi^2 \frac{M}{2(M+1)}
+ O(x^{1/2})$ of these are divisible by 2. Thus, if our family is
all even fundamental discriminants, we do get the factor of
$1/(p+1)$ for $p=2$, as one-third (which is $1/(2+1)$ of the
fundamental discriminants in this family are divisible by $2$.
\end{proof}

In our analysis of the terms from the $L$-functions Ratios
Conjecture, we shall need a partial summation consequence of Lemma
\ref{lem:numdwithpdivided}.

\begin{lem}\label{lem:partialsumdexpdpiX} Let $\mathcal{F}(X)$ denote all even fundamental discriminants congruent to a non-zero square modulo $M$
that are at most $X$, and set $X^\ast = \sum_{d \in \mathcal{F}(X)} 1$. Let $z = \tau - i w \frac{L}{2\pi}$ with $w \in [0, 1/2]$ and $L = \log(\sqrt{M}X/2\pi)$.
Then \be \sum_{d\in\mathcal{F}(X)} \left(\frac{\sqrt{M}d}{2\pi}\right)^{-\frac{2\pi i z}{L}} \ = \
X^\ast e^{-2\pi i z} \left(1
- \frac{2\pi i z}{L}\right)^{-1} + O(X^{1/2-w}\log X). \ee
\end{lem}

\begin{proof} Note \bea \sum_{d\in\mathcal{F}(X)} \left(\frac{\sqrt{M}d}{2\pi}\right)^{-\frac{2\pi i z}{L}} & \ = \ & \sum_{d\in\mathcal{F}(X)} \exp\left(-2\pi i z \frac{\sqrt{M}/2\pi}{L}\right) \exp\left(-\frac{2\pi i z}{L} \log d\right) \nonumber\\ &=& \exp\left(-2\pi i z + 2\pi i z \frac{\log X}{L}\right) \sum_{d\in\mathcal{F}(X)} d^{-2\pi i z/L}. \eea We now analyze $\sum_{d \in \mathcal{F}(X)} d^{-2\pi i z/L}$. By Lemma
\ref{lem:numdwithpdivided} we have \be \sum_{d \in \mathcal{F}(u)} 1 \ = \
\frac{3u}{\pi^2} \frac{M}{2(M+1)} + O(u^{1/2}). \ee Therefore by partial summation we
have \bea \sum_{d\in \mathcal{F}(X)} d^{-2\pi i z/L} & \ = \ & \left(X^\ast + O(X^{1/2})\right) X^{-\frac{2\pi i z}{L}}\nonumber\\ & & \ \  - \  \int_1^X \left(\frac{3u}{\pi^2} \frac{M}{2(M+1)} + O(u^{1/2})\right) u^{-\frac{2\pi i z}{L}} \frac{-2\pi i z}{L} \frac{du}{u}. \nonumber\\ \eea As $w \in [0, 1/2]$, the error terms contribute at most $O(X^{1/2-w}\log X)$ (we need to add the $\log X$ as if $w = 1/2$ the integral of the error is $\log X$); further, we may absorb the lower boundary term of the integral in the $O(X^{1/2-w}\log X)$ error term, and we find \bea & &  \sum_{d\in \mathcal{F}(X)} d^{-2\pi i z/L} \nonumber\\ & &  = \  X^\ast \exp\left(-\frac{2\pi i z \log X}{L}\right) + \frac{3}{\pi^2} \frac{M}{2(M+1)} \frac{X^{1 - \frac{2\pi i z}{L}}}{1 - \frac{2\pi i z}{L}} + O(X^{1/2-w}\log X) \nonumber\\ & & = \ X^\ast \exp\left(-\frac{2\pi i z \log X}{L}\right) + X^\ast \exp\left(-\frac{2\pi i z \log X}{L}\right) + \frac{2\pi i z}{L} \sum_{\nu=0}^\infty \left(\frac{2\pi i z}{L}\right)^\nu \nonumber\\ & & \ \ \ \ \ \ +\ O(X^{1/2-w}\log X) \nonumber\\ & & = \ X^\ast \exp\left(-\frac{2\pi i z \log X}{L}\right) \left(1 - \frac{2\pi i z}{L}\right)^{-1} + O(X^{1/2-w}\log X). \eea
Substituting yields the claim.
\end{proof}


\section{Schwartz function expansions}

Let $\phi$ be an even Schwartz
function and $\hphi$ be its Fourier transform ($\hphi(\xi) = \int
\phi(x) e^{-2\pi i x \xi}dx$); we often assume ${\rm supp}(\hphi)
\subset (-\sigma, \sigma)$ for some $\sigma < \infty$. We set \be
H(s) \ = \ \phi\left(\frac{s - \foh}{i}\right). \ee While $H(s)$ is
initially define only when $\Re(s) = 1/2$, because of the compact
support of $\hphi$ we may extend it to all of $\C$:
\bea\label{eq:extendingphi}
\phi(x) &\ = \ & \intii \hphi(\xi) e^{2\pi i x\xi} d\xi \nonumber\\
\phi(x+iy) &=& \intii \hphi(\xi) e^{2\pi i(x+iy)\xi} d\xi
\nonumber\\ H(x+iy) &=& \intii \left[\hphi(\xi)e^{2\pi(x -
\foh)}\right] \cdot e^{2\pi i y \xi} d\xi. \eea Note that $H(x+iy)$
is rapidly decreasing in $y$ (for a fixed $x$ it is the Fourier
transform of a nice function, and thus the claim follows from the
Riemann-Lebesgue lemma).

The following result is useful in expanding some terms in the Ratios' prediction.

\begin{lem}\label{lem:usefulfacts} Let $\supp(\hg) \subset (-\sigma, \sigma) \subset (-1,1)$ and $L$ $=$ $\log(\sqrt{M}X/2\pi)$. \ben
\item For $w \ge 0$, $g\left(\tau-iw \frac{L}{2\pi}\right) \ll X^{\sigma w} \left(\tau^2 + (w\frac{L}{2\pi})^2\right)^{-B}$ for any $B \ge 0$. \item For $0 < a < b$ we
have $|\Gamma(a\pm i y)/\Gamma(b\pm i y)| = O_{a,b}(1)$. \een
\end{lem}

\begin{proof}
(1): As $g(\tau) = \int \hg(\xi) e^{2\pi i \xi \tau} d\xi$, we have
\bea g(\tau-iy) & \ = \ & \intii \hg(\xi) e^{2\pi i (\tau - iy)\xi}
d\xi \nonumber\\ & = & \intii \hg^{(2n)}(\xi) (2\pi i (\tau
-iy))^{-n} e^{2\pi i (\tau - iy)\xi} d\xi \nonumber\\ & \ll &
e^{2\pi y \sigma} (\tau -iy))^{-2n}; \eea the claim follows by
taking $y = wL/2\pi$.

(2): As $|\Gamma(x-iy)|=|\Gamma(x+iy)|$, we may assume all signs are
positive. The claim follows from the definition of the Beta
function: \be \frac{\Gamma(a+iy)\Gamma(b-a)}{\Gamma(b+iy)} \ = \
\int_0^1 t^{a+iy-1} (1-t)^{b-a-1} \ = \ O_{a,b}(1). \ee
\end{proof}



\ \\

\end{document}